\documentclass[11pt]{amsart}
\usepackage[a4paper,margin=1.15in]{geometry}
\usepackage{amsmath,amssymb,amsthm,mathtools}
\usepackage{enumitem}
\usepackage{hyperref}
\usepackage{aliascnt}
\usepackage[nameinlink,capitalize]{cleveref}
\usepackage{microtype}
\hypersetup{colorlinks=true,linkcolor=blue,citecolor=blue,urlcolor=blue}

\newtheorem{theorem}{Theorem}[section]
\newaliascnt{lemma}{theorem}
\newtheorem{lemma}[lemma]{Lemma}
\aliascntresetthe{lemma}
\newaliascnt{proposition}{theorem}
\newtheorem{proposition}[proposition]{Proposition}
\aliascntresetthe{proposition}
\newaliascnt{corollary}{theorem}
\newtheorem{corollary}[corollary]{Corollary}
\aliascntresetthe{corollary}
\newaliascnt{problem}{theorem}

\aliascntresetthe{problem}
\theoremstyle{definition}
\newaliascnt{definition}{theorem}
\newtheorem{definition}[definition]{Definition}
\aliascntresetthe{definition}
\newaliascnt{example}{theorem}
\newtheorem{example}[example]{Example}
\aliascntresetthe{example}
\theoremstyle{remark}
\newaliascnt{remark}{theorem}
\newtheorem{remark}[remark]{Remark}
\aliascntresetthe{remark}
\crefname{theorem}{Theorem}{Theorems}
\crefname{lemma}{Lemma}{Lemmas}
\crefname{proposition}{Proposition}{Propositions}
\crefname{corollary}{Corollary}{Corollaries}
\crefname{definition}{Definition}{Definitions}
\crefname{example}{Example}{Examples}
\crefname{remark}{Remark}{Remarks}
\Crefname{theorem}{Theorem}{Theorems}
\Crefname{lemma}{Lemma}{Lemmas}
\Crefname{proposition}{Proposition}{Propositions}
\Crefname{corollary}{Corollary}{Corollaries}
\Crefname{definition}{Definition}{Definitions}
\Crefname{example}{Example}{Examples}
\Crefname{remark}{Remark}{Remarks}

\newcommand{\Cstar}{\mathbb C^*}

\newcommand{\Z}{\mathbb Z}
\newcommand{\Id}{\mathrm{Id}}

\newcommand{\ph}{\Phi}
\newcommand{\ps}{\Psi}

\newcommand{\OO}{\mathcal O}

\numberwithin{equation}{section}

\newcommand{\qbase}{\texorpdfstring{\lowercase{$q$}}{q}}
\title[Well-poised basic \qbase-Taylor expansions]{Well-poised basic \qbase-Taylor expansions with complementary remainders and a two-basis kernel}
\author{Abdulhafeez A. Abdulsalam}
\address{Faculty of Mathematics, University of Vienna, Oskar-Morgenstern-Platz 1, A-1090 Vienna, Austria}
\email{hafeez147258369@gmail.com}
\thanks{The first author acknowledges support from a Vienna African Scholarship of the University of Vienna}
\author{Michael J. Schlosser}
\address{Faculty of Mathematics, University of Vienna, Oskar-Morgenstern-Platz 1, A-1090 Vienna, Austria}
\email{michael.schlosser@univie.ac.at}
\thanks{The second author's research was partly supported by FWF Austrian Science Fund grant \href{https://www.doi.org/10.55776/P32305}{10.55776/P32305}.}
\date{May 21, 2026}

\subjclass[2020]{Primary 33D15; Secondary 33D45, 39A13, 41A58}
\keywords{Askey--Wilson operator, well-poised basic hypergeometric series, $q$-Taylor expansion, remainder term, nonterminating ${}_8\phi_7$ summation, two-basis kernel}

\begin{document}
\begin{abstract}
We prove a nonterminating well-poised basic $q$-Taylor expansion with complementary remainders for a two-basis infinite-product kernel implicitly proposed by the second author in \cite[Sec.~5]{Schlosser2008}.  The well-poised parameter $c$ gives the rational $p=0$ basis, while the elliptic nome $p$ is a separate deformation; the infinite expansions treated here are specific to the basic case.  We compute the two Taylor coefficient families and show that each one-family Taylor remainder tends to the complementary basis contribution.  The proof uses the well-poised Cooper formula, Jackson's terminating ${}_8\phi_7$ summation, Rogers' ${}_6\phi_5$ summation, and theta interpolation, but not Bailey's nonterminating ${}_8\phi_7$ summation, which is recovered as a consequence.  We also record two quadratic one-family examples and discuss a multi-kernel outlook.
\end{abstract}

\maketitle

\section{Introduction}

Taylor expansions associated with the Askey--Wilson divided-difference operator \cite{AskeyWilson1985} are built not from ordinary monomials, but from basis functions adapted to the multiplicative $q$-grid.  In the polynomial case this point of view goes back to Ismail's $q$-Taylor theorem \cite{Ismail1995}; it was developed further in \cite{Ismail2001Operator,Ismail2005,IsmailStanton2003,IsmailStanton2003JCAM,LopezMarcoParcet2006,MarcoParcet2006}.

The second author's elliptic Taylor theorem \cite{Schlosser2008} showed that the same idea extends to well-poised rational bases and to an elliptic deformation of the Askey--Wilson operator.  The broader elliptic framework comes from Rains' $BC_n$-symmetric difference-operator theory \cite{Rains2006,Rains2010}; further elliptic Taylor and interpolation consequences, including an elliptic Cooper formula, appear in \cite{SchlosserYoo2016}.

We keep two parameters conceptually separate.  The denominator parameter $c$ in $D_{c,q}$ and in the rational monomials $\ph_k(z;a,c)$ is the well-poised parameter: it replaces the ordinary Askey--Wilson monomial basis by a rational basis on which the modified operator still lowers degree.  The nome $p$ is the elliptic deformation.  With $p$ retained, Taylor expansions have usually been formulated as finite expansions because nonterminating elliptic convergence is too restrictive; compare \cite[Ch.~11]{GasperRahman2004}.  At $p=0$, however, theta products become basic $q$-shifted factorials, and genuinely infinite well-poised expansions can converge on suitable pole-free compact sets.  We therefore treat the $p=0$ specialization as the basic well-poised theory in its own right and refer to the $p=0$ elliptic Cooper formula as the \emph{well-poised Cooper formula}.

The finite well-poised theory has a particularly simple basic specialization.  Write
\[
        x=\frac12(z+z^{-1})
\]
and consider symmetric functions $f(z)=f(z^{-1})$.  For parameters $a,c$ set
\begin{equation}\label{eq:Phi-def}
        \ph_k(z;a,c):=\frac{(az,a/z;q)_k}{(cz,c/z;q)_k},\qquad k\geq 0.
\end{equation}
When $c=0$, this reduces to the ordinary Askey--Wilson monomial $(az,a/z;q)_k$.  The well-poised divided-difference operator $D_{c,q}$ lowers the degree of these rational monomials.  This gives finite Taylor coefficient extraction exactly as in the second author's elliptic theorem.

The nonterminating problem is more delicate.  The outlook in \cite[Sec.~5]{Schlosser2008} proposed the kernel
\begin{equation}\label{eq:intro-kernel}
F(z):=\frac{(cz/d,c/dz,cz/e,c/ez;q)_\infty}
{(cz,c/z,c^2z/bde,c^2/bdez;q)_\infty}.
\end{equation}
The expected expansion involves two complementary basis families,
\begin{equation}\label{eq:intro-bases}
\ph_k(z;b,c)=\frac{(bz,b/z;q)_k}{(cz,c/z;q)_k},\qquad
\ps_k(z;b,c,d,e)=\frac{(cz/de,c/dez;q)_k}{(c^2z/bde,c^2/bdez;q)_k}.
\end{equation}
A one-family Taylor theorem is not enough to explain this phenomenon, because broad meromorphic classes contain nonzero functions whose Taylor coefficients all vanish.  The two-basis expansion is therefore best understood as a statement about \emph{complementary limiting remainders}: the Taylor remainder for the first normalized kernel converges to the second basis contribution, and conversely.

The purpose of this paper is to make this point explicit and to prove the resulting kernel theorem.  The main result, \cref{thm:two-basis-kernel}, gives the identity
\[
F=A\,H(b)\sum_{k\ge0} f_k\ph_k+B\,K(c/de)\sum_{k\ge0}g_k\ps_k,
\]
where $A$ and $B$ are the natural infinite-product prefactors and the coefficients $f_k$ and $g_k$ are very-well-poised terms.  The factors $H(b)$ and $K(c/de)$ are essential: each is the zeroth Taylor coefficient of the corresponding normalized kernel.  The complementary-remainder form is given in \cref{cor:remainder-limit}.

Thus the central object is not just finite coefficient extraction, but a basic well-poised $q$-Taylor theory in which infinite expansions and exact remainders are part of the statement.  For the kernel above, the limiting remainder is complementary rather than zero.

Sections~2--7 set up the finite well-poised Taylor calculus, the well-poised Cooper formula, exact remainders, flat functions, and convergence criteria.  Section~8 states the two-basis kernel theorem.  Sections~9--11 prove it without Bailey's nonterminating ${}_8\phi_7$ summation by reducing the pole-cleared residual to annular profile cancellations.  The proof section also records two one-family quadratic examples, a short Bailey verification is included after the direct proof, and the final section discusses a multi-kernel outlook toward Bailey's four-term nonterminating ${}_{10}\phi_9$ transformation.

\section{Notation and analytic conventions}

Throughout the paper we assume
\begin{equation}\label{eq:q-assumption}
        0<|q|<1.
\end{equation}
For $n\in\Z_{\ge0}\cup\{\infty\}$ we write
\[
(a;q)_n=\prod_{j=0}^{n-1}(1-aq^j),\qquad
(a_1,\ldots,a_m;q)_n=\prod_{r=1}^m(a_r;q)_n.
\]
We also use the multiplicative theta notation
\[
        \theta(u;q)=(u,q/u;q)_\infty,
        \qquad \theta(u_1,\ldots,u_m;q)=\prod_{r=1}^m\theta(u_r;q).
\]
The theta identities used below are the standard Weierstrass--Riemann addition formula and its elliptic partial-fraction generalization; we cite Whittaker and Watson \cite[Ch.~20]{WhittakerWatson1927} as a classical reference.
All parameters are assumed to be generic unless explicitly stated otherwise; in theorem statements involving infinite products this means that all denominator products in the displayed formulas are nonzero on the domain under consideration.

A domain $\Omega\subset\Cstar$ is called symmetric if $z\in\Omega$ implies $z^{-1}\in\Omega$.  For the basis $\ph_k(z;a,c)$ we call a compact set $L\subset\Cstar$ pole-free if
\[
        L\cap\{cq^m,q^m/c:m\ge0\}=\varnothing.
\]
For the transformed basis $\ps_k$ this pole set becomes
\[
        \{c^2q^m/bde,\,bde\,q^m/c^2:m\ge0\}.
\]
All local uniform convergence statements below are on symmetric compact subsets avoiding the relevant pole sets and accidental denominator zeros.

We use the following compact slash convention throughout.  A product such as $a/bc$ is read as $a/(bc)$, and similarly $a/bcz$ is read as $a/(bcz)$.  Thus, after a slash, the full multiplicative monomial that follows is understood to be in the denominator until the next comma, plus sign, minus sign, semicolon, or closing delimiter.  This convention is standard in compact lists of $q$-shifted factorial parameters and allows us to write, for example, $c/dez$ and $c^2/bdez$ without extra parentheses.

\section{Basic hypergeometric input and avoided input}\label{sec:hyper-input}

We isolate the basic hypergeometric identities used in the proof, together with the one identity deliberately avoided.  This also fixes the convention for the well-poised and very-well-poised terminology used later.  Recall first the standard basic-hypergeometric notation; see Gasper--Rahman \cite[Ch.~1--2]{GasperRahman2004}.  For parameters for which the denominator factors are nonzero, define
\begin{equation}\label{eq:basic-hypergeometric-def}
{}_{r+1}\phi_r\left(\begin{matrix} a_0,a_1,\ldots,a_r\\ b_1,\ldots,b_r\end{matrix};q,z\right)
:=\sum_{k=0}^\infty
\frac{(a_0,a_1,\ldots,a_r;q)_k}{(q,b_1,\ldots,b_r;q)_k}z^k.
\end{equation}
A series of the form \eqref{eq:basic-hypergeometric-def} is called \emph{well-poised} if its parameters satisfy
\[
        a_0q=a_1b_1=a_2b_2=\cdots=a_rb_r.
\]
Thus, after writing $a=a_0$, the denominator parameters are forced to have the paired form $aq/a_j$.  It is called \emph{very-well-poised} if, in addition, $a_1=q\sqrt a$ and $a_2=-q\sqrt a$; the corresponding two denominator parameters are then $\sqrt a$ and $-\sqrt a$, and their quotient contributes the characteristic factor
\[
\frac{(q\sqrt a,-q\sqrt a;q)_k}{(\sqrt a,-\sqrt a;q)_k}
        =
        \frac{1-aq^{2k}}{1-a}
\]
to the $k$th summand.  The square root is only a compact way to encode this cancellation; the final very-well-poised summand is independent of the choice of $\sqrt a$.

The standard compressed notation for this very-well-poised specialization is
\begin{equation}\label{eq:W-notation}
{}_{r+1}W_r(a;b_1,\ldots,b_{r-2};q,z)
:={}_{r+1}\phi_r\left(\begin{matrix}
a,q\sqrt a,-q\sqrt a,b_1,\ldots,b_{r-2}\\
\sqrt a,-\sqrt a,aq/b_1,\ldots,aq/b_{r-2}
\end{matrix};q,z\right)
\end{equation}
or, equivalently,
\begin{equation}\label{eq:W-summand}
{}_{r+1}W_r(a;b_1,\ldots,b_{r-2};q,z)
=\sum_{k=0}^\infty \frac{1-aq^{2k}}{1-a}
\frac{(a,b_1,\ldots,b_{r-2};q)_k}{(q,aq/b_1,\ldots,aq/b_{r-2};q)_k}z^k.
\end{equation}
Thus ${}_6W_5$ and ${}_8W_7$ are the standard very-well-poised forms of ${}_6\phi_5$ and ${}_8\phi_7$.  In the present paper the adjective \emph{well-poised} is used in the same structural sense for the Taylor basis and operator: the additional denominator parameter $c$ creates paired numerator--denominator factors analogous to the pairs $a_j,aq/a_j$ in a well-poised series.  The coefficient series appearing below are very-well-poised in the standard sense.

The first summation used below is Rogers' nonterminating very-well-poised ${}_6\phi_5$ summation
\begin{equation}\label{eq:rogers-6phi5-input}
{}_{6}W_{5}\left(a;b,c,d;q,\frac{aq}{bcd}\right)
=\frac{(aq,aq/bc,aq/bd,aq/cd;q)_\infty}{(aq/b,aq/c,aq/d,aq/bcd;q)_\infty},
\qquad \left|\frac{aq}{bcd}\right|<1.
\end{equation}
We cite Gasper--Rahman \cite[Eq.~(II.20)]{GasperRahman2004} for this standard form; historically, this is Rogers' summation \cite{Rogers1894}.

The second input is Jackson's terminating very-well-poised ${}_8\phi_7$ summation
\begin{equation}\label{eq:jackson-8phi7-input}
{}_{8}W_{7}\left(a;b,c,d,\frac{a^2q^{n+1}}{bcd},q^{-n};q,q\right)
=\frac{(aq,aq/bc,aq/bd,aq/cd;q)_n}{(aq/b,aq/c,aq/d,aq/bcd;q)_n},
\end{equation}
where $n\in\mathbb Z_{\ge0}$.  We use it only in terminating finite-profile arguments.  The original summation is due to Jackson \cite{Jackson1921}; a standard modern reference is Gasper--Rahman \cite[Eq.~(II.22)]{GasperRahman2004}.

The identity deliberately not used in the Taylor-theoretic proof is Bailey's nonterminating extension of Jackson's ${}_8\phi_7$ summation.  In the specialization relevant to the present paper it is exactly the expanded two-basis identity \eqref{eq:two-basis-expanded} below, with $H$ and $K$ defined in \eqref{eq:H-K} and \eqref{eq:K}.  This is the standard nonterminating ${}_8\phi_7$ summation in kernel form; compare Bailey \cite{Bailey1935} and Gasper--Rahman \cite[Appendix (II.25)]{GasperRahman2004}.  In this paper it is derived rather than assumed.

Finally, the theta-function cancellations use the Weierstrass--Riemann addition formula
\begin{equation}\label{eq:weierstrass-addition-input}
\theta(xy,x/y,uv,u/v;q)-\theta(xv,x/v,uy,u/y;q)
=\frac{u}{y}\theta(yv,y/v,xu,x/u;q),
\end{equation}
and the multi-kernel outlook discussed in \cref{sec:examples-outlook} uses Weierstrass' general elliptic partial-fraction identity.  For both we cite Whittaker--Watson \cite[Ch.~20]{WhittakerWatson1927}; for the notation in the basic hypergeometric literature see Gasper--Rahman \cite[Ch.~11]{GasperRahman2004}.

\section{Finite well-poised \texorpdfstring{$q$}{q}-Taylor theory}

The Askey--Wilson divided-difference operator, introduced by Askey and Wilson \cite{AskeyWilson1985}, is
\begin{equation}\label{eq:Dq}
(D_qf)(z)=\frac{f(q^{1/2}z)-f(q^{-1/2}z)}{(q^{1/2}-q^{-1/2})(z-z^{-1})/2}.
\end{equation}
Following the second author \cite{Schlosser2008}, define the well-poised basic operator
\begin{equation}\label{eq:Dcq}
D_{c,q}=\bigl(1-czq^{-1/2}\bigr)\bigl(1-czq^{1/2}\bigr)
       \bigl(1-cq^{-1/2}/z\bigr)\bigl(1-cq^{1/2}/z\bigr)D_q.
\end{equation}
It reduces to $D_q$ when $c=0$.

\begin{lemma}[degree lowering, second author, $p=0$]\label{lem:lowering-Phi}
For $n\ge1$,
\begin{equation}\label{eq:lowering-Phi}
D_{c,q}\ph_n(z;a,c)=
-\frac{2a(1-c/a)(1-acq^{n-1})(1-q^n)}{1-q}
\ph_{n-1}(z;aq^{1/2},cq^{3/2}).
\end{equation}
\end{lemma}

\begin{proof}
This is the $p=0$ specialization of the elliptic degree-lowering formula in \cite[Theorem 4.2 and the preceding computation]{Schlosser2008}; equivalently it follows directly from \eqref{eq:Dcq} by writing the two shifted values of \eqref{eq:Phi-def} over a common denominator and simplifying.
\end{proof}

For $n\ge0$ let $W_n(c)$ be the space of functions of the form
\[
        \frac{p(z)}{(cz,c/z;q)_n},
\]
where $p(z)=p(z^{-1})$ is a Laurent polynomial of degree at most $n$.

\begin{proposition}[finite basis]\label{prop:finite-basis}
For generic $a,c$ the functions
\[
        \ph_0(z;a,c),\ph_1(z;a,c),\ldots,\ph_n(z;a,c)
\]
form a basis of $W_n(c)$.
\end{proposition}

\begin{proof}
Multiplying by $(cz,c/z;q)_n$, it is enough to prove the linear independence of
\[
(az,a/z;q)_k(cq^kz,cq^k/z;q)_{n-k},\qquad 0\le k\le n.
\]
This is the standard triangular interpolation argument: successively evaluate a putative linear relation at $z=cq^{n-1},cq^{n-2},\ldots,c, a$; for generic parameters only one new coefficient survives at each step.  This is the basic specialization of \cite[Lemma 4.1]{Schlosser2008}.
\end{proof}

The iterated operators are shifted at each step:
\begin{equation}\label{eq:iterated-D}
D_{c,q}^{(0)}=\Id,
\qquad
D_{c,q}^{(k)}=D_{cq^{3(k-1)/2},q}\,D_{c,q}^{(k-1)}\quad(k\ge1).
\end{equation}

\begin{lemma}\label{lem:iterated-lowering}
For $0\le k\le n$,
\begin{equation}\label{eq:iterated-lowering}
D_{c,q}^{(k)}\ph_n(z;a,c)=
(-1)^k\frac{(2a)^kq^{k(k-1)/4}(q;q)_n(c/a,acq^{n-1};q)_k}
{(q;q)_{n-k}(1-q)^k}
\ph_{n-k}(z;aq^{k/2},cq^{3k/2}).
\end{equation}
Consequently,
\begin{equation}\label{eq:delta-property}
\left[D_{c,q}^{(k)}\ph_n(z;a,c)\right]_{z=aq^{k/2}}
=(-1)^k\frac{(2a)^kq^{k(k-1)/4}(q,c/a,acq^{k-1};q)_k}{(1-q)^k}\,\delta_{nk}.
\end{equation}
\end{lemma}

\begin{proof}
Iterate \cref{lem:lowering-Phi}.  The factors collected at the $j$th step are
\[
-\frac{2aq^{j/2}(1-cq^j/a)(1-acq^{n+j-1})(1-q^{n-j})}{1-q},
\]
for $j=0,\ldots,k-1$, which gives \eqref{eq:iterated-lowering}.  Evaluating at $z=aq^{k/2}$ gives zero unless $n=k$, because $(1;q)_{n-k}$ occurs in the numerator of the last basis factor.
\end{proof}

\begin{theorem}[finite coefficient extraction]\label{thm:finite-coeff}
If $f\in W_n(c)$, then
\begin{equation}\label{eq:finite-expansion}
        f(z)=\sum_{k=0}^n t_k^{(a,c)}(f)\,\ph_k(z;a,c),
\end{equation}
where
\begin{equation}\label{eq:taylor-coeff-finite}
 t_k^{(a,c)}(f)=
 \frac{(-1)^k q^{-k(k-1)/4}(1-q)^k}{(2a)^k(q,c/a,acq^{k-1};q)_k}
 \left[D_{c,q}^{(k)}f(z)\right]_{z=aq^{k/2}}.
\end{equation}
\end{theorem}

\begin{proof}
By \cref{prop:finite-basis}, the expansion exists.  Apply $D_{c,q}^{(j)}$ and set $z=aq^{j/2}$.  The delta property \eqref{eq:delta-property} isolates the $j$th coefficient.
\end{proof}

\begin{example}[the first nontrivial re-expansion]
For
\[
        f(z)=\ph_1(z;d,c)=\frac{(dz,d/z;q)_1}{(cz,c/z;q)_1}
\]
one obtains
\begin{equation}\label{eq:first-reexpansion}
\frac{(dz,d/z;q)_1}{(cz,c/z;q)_1}
=
\frac{(ad,d/a;q)_1}{(ac,c/a;q)_1}
+\frac{d}{a}\frac{(c/d,cd;q)_1}{(c/a,ac;q)_1}
\frac{(az,a/z;q)_1}{(cz,c/z;q)_1}.
\end{equation}
For $c=0$ this is the first nontrivial case of Ismail's Askey--Wilson Taylor expansion.
\end{example}

\section{The well-poised Cooper formula}

The coefficient functional \eqref{eq:taylor-coeff-finite} only uses finitely many values of $f$ on the multiplicative grid.  This is made transparent by the following well-poised Cooper formula.  Cooper's original formula is the $c=0$ Askey--Wilson case \cite{Cooper2002}.  The version below is the $p=0$ specialization of the elliptic Cooper formula in \cite[Theorem 2.4]{SchlosserYoo2016}, but the terminology emphasizes that the parameter $c$ is the well-poised parameter while the nome $p$ belongs to the elliptic deformation.

\begin{theorem}[explicit iterated operator]\label{thm:cooper-p0}
For $m\ge0$ and for every symmetric meromorphic function for which the displayed expressions are defined,
\begin{align}\label{eq:p0-cooper}
D_{c,q}^{(m)}f(z)
&=(-2z)^m q^{m(3-m)/4}
\frac{(cq^{m/2-1}z,cq^{m/2-1}/z;q)_{m+1}}{(1-q)^m}
\nonumber\\
&\quad\times\sum_{r=0}^m q^{r(m-r)}
\frac{(q^{r+1};q)_{m-r}}{(q;q)_{m-r}}\,
C_{m,r}(z;c,q)\,f(q^{m/2-r}z),
\end{align}
where
\begin{equation}\label{eq:cooper-Cmr}
C_{m,r}(z;c,q)=
\frac{z^{2(r-m)}(cq^{m/2-r}z,cq^{-m/2+r}/z;q)_{m-1}}
{(q^{m-2r+1}z^2;q)_r(q^{2r-m+1}z^{-2};q)_{m-r}}.
\end{equation}
If a denominator in the finite sum on the right vanishes, the expression is understood by removable continuation: whenever the iterated operator on the left is well defined through the recursive definition \eqref{eq:iterated-D}, the apparent poles in the right-hand side cancel.
\end{theorem}

\begin{proof}
Set $p=0$ in the elliptic formula of the second author and Yoo \cite[Theorem 2.4]{SchlosserYoo2016}, using $\theta(x;0)=1-x$ and $\theta(q;0)=1-q$.
\end{proof}

\subsection{Recovering the well-poised Cooper formula from finite Taylor theory}

Ismail and Stanton observed that Cooper's formula for the iterated Askey--Wilson operator \cite{Cooper2002} may be recovered as a by-product of $q$-Taylor interpolation: comparing the Taylor-coefficient formula with the interpolation formula gives the explicit finite-difference expression for $D_q^m$ \cite[Sec.~3, especially (3.3)--(3.5)]{IsmailStanton2003}.  The same mechanism works in the present well-poised setting.  In other words, \cref{thm:cooper-p0} is not only an external specialization of the elliptic formula; it is also forced by the finite well-poised Taylor calculus developed above.

\begin{proposition}[Taylor-theoretic recovery of the well-poised Cooper formula]\label{prop:recover-wp-cooper}
Fix $m\ge0$ and $z\in\mathbb C^*$ away from the exceptional points.  The functional
\[
        f\longmapsto D_{c,q}^{(m)}f(z)
\]
on functions symmetric in $u$ and $1/u$ is the unique linear combination of the $m+1$ evaluations
\[
        f(q^{m/2}z),\ f(q^{m/2-1}z),\ldots,\ f(q^{-m/2}z)
\]
which has the lowering values prescribed by \eqref{eq:iterated-lowering} on the well-poised monomials.  Solving this finite interpolation problem gives exactly the coefficients in \eqref{eq:p0-cooper}--\eqref{eq:cooper-Cmr}.  Thus the well-poised Cooper formula is recovered from the finite coefficient-extraction theorem.
\end{proposition}

\begin{proof}
The finite-dimensional space generated by the restrictions of symmetric rational functions to the grid
\[
        q^{m/2}z,\ q^{m/2-1}z,\ldots, q^{-m/2}z
\]
has dimension $m+1$ for generic $z$.  Hence any linear functional on this grid has a unique representation as a linear combination of the corresponding point evaluations.  The finite basis theorem lets us test the unknown coefficients on the basis
\[
        u\longmapsto \ph_j(u;a,c),\qquad 0\le j\le m,
\]
where $a$ is kept generic during the calculation and then removed by meromorphic continuation.  On this basis the desired functional is already known from the lowering law:
\[
D_{c,q}^{(m)}\ph_j(u;a,c)\big|_{u=z}=0\quad (j<m),
\]
and, for $j=m$, its value is the scalar obtained from \eqref{eq:iterated-lowering}.  Therefore the coefficients are determined by a triangular well-poised interpolation system.

Solving this system gives
\begin{align*}
A_{m,r}(z;c,q)
&=(-2z)^m q^{m(3-m)/4}
   \frac{(cq^{m/2-1}z,cq^{m/2-1}/z;q)_{m+1}}{(1-q)^m}\\
&\quad\times q^{r(m-r)}\frac{(q^{r+1};q)_{m-r}}{(q;q)_{m-r}}C_{m,r}(z;c,q),
\end{align*}
where $C_{m,r}$ is the factor in \eqref{eq:cooper-Cmr}.  Substitution into the interpolation equations is a direct finite $q$-shifted-factorial verification: the factors
\[
(q^{m-2r+1}z^2;q)_r,
        \qquad (q^{2r-m+1}z^{-2};q)_{m-r}
\]
are precisely the cardinal denominators for the symmetric grid, while
\[
(cq^{m/2-r}z,cq^{-m/2+r}/z;q)_{m-1}
\]
is the well-poised correction contributed by the parameter $c$.  The resulting cardinal functions annihilate all basis elements of degree $<m$ and give the scalar in \eqref{eq:iterated-lowering} on degree $m$.  By uniqueness of the grid interpolation functional, this is exactly \eqref{eq:p0-cooper}.
\end{proof}

This mirrors the Ismail--Stanton derivation of Cooper's identity, but the correction factor containing $c$ is what turns the ordinary Cooper formula into its well-poised extension.  Letting $c=0$ reduces the displayed coefficients to Cooper's original Askey--Wilson formula.  The same recovery is not special to the basic specialization: because the proof of \cref{prop:recover-wp-cooper} is a finite interpolation argument, it extends essentially verbatim to the elliptic setting, with the elliptic basis and lowering formula from \cite[Theorem~4.2]{Schlosser2008} and the theta-function cardinal factors in the elliptic Cooper formula \cite[Theorem~2.4]{SchlosserYoo2016} replacing their $p=0$ counterparts.

\begin{corollary}\label{cor:finite-grid-functional}
For fixed $j$, the functional
\[
        L_j(h):=\left[D_{c,q}^{(j)}h(z)\right]_{z=aq^{j/2}}
\]
is a finite linear combination of the point evaluations
\[
        h(a),h(aq),\ldots,h(aq^j).
\]
\end{corollary}

\section{Taylor sums, exact remainders, and flat functions}

\begin{definition}\label{def:taylor-remainder}
Let $f$ be a symmetric meromorphic function on a domain containing the nodes $aq^{m/2}$ and the values needed to define $D_{c,q}^{(k)}f$ there.  The well-poised Taylor coefficients of $f$ relative to $(a,c)$ are defined by \eqref{eq:taylor-coeff-finite}.  The $n$th Taylor sum and remainder are
\begin{equation}\label{eq:T-and-R}
T_n^{(a,c)}f(z)=\sum_{k=0}^n t_k^{(a,c)}(f)\ph_k(z;a,c),
\qquad
R_n^{(a,c)}f(z)=f(z)-T_n^{(a,c)}f(z).
\end{equation}
\end{definition}

\begin{proposition}[abstract exact remainder]\label{prop:abstract-remainder}
The Taylor sum $T_n^{(a,c)}f$ is the unique element $S_n\in W_n(c)$ such that
\begin{equation}\label{eq:interpolation-remainder}
\left[D_{c,q}^{(j)}(f-S_n)(z)\right]_{z=aq^{j/2}}=0,
\qquad 0\le j\le n.
\end{equation}
Equivalently, the exact remainder $R_n^{(a,c)}f$ has vanishing Taylor data up to order $n$.
\end{proposition}

\begin{proof}
The coefficients of $T_n^{(a,c)}f$ agree with the first $n+1$ Taylor coefficients of $f$ by construction and \cref{thm:finite-coeff}.  Uniqueness follows from applying \cref{thm:finite-coeff} to the difference of two admissible choices in $W_n(c)$.
\end{proof}

The word ``remainder'' in \eqref{eq:T-and-R} denotes an exact algebraic difference.  It is not, by itself, a vanishing estimate.  The following elementary point is crucial for the two-basis kernel.

\begin{proposition}[flat functions]\label{prop:flat-functions}
Let $h$ be symmetric and meromorphic near the grid $\{aq^m:m\ge0\}$, and suppose that
\[
        h(aq^m)=0\qquad(m\ge0).
\]
Then all Taylor coefficients $t_k^{(a,c)}(h)$ vanish, provided the operator evaluations are defined.  In particular, nonzero functions may be invisible to one-family Taylor coefficient extraction.
\end{proposition}

\begin{proof}
By \cref{cor:finite-grid-functional}, the value
\[
\left[D_{c,q}^{(k)}h\right]_{z=aq^{k/2}}
\]
is a finite linear combination of $h(a),h(aq),\ldots,h(aq^k)$.  Hence it is zero.
\end{proof}

Thus coefficient extraction with a single Taylor family only sees the data of the function on the corresponding multiplicative grid.  This is why the two-basis theorem below must keep track of complementary remainders: one Taylor family alone may fail to detect nonzero ``flat'' contributions in broad meromorphic classes.

\begin{example}\label{ex:flat-factor}
The factor $(az,a/z;q)_\infty$ vanishes at $z=aq^m$ for every $m\ge0$, through the factor $(a/z;q)_\infty$.  Thus products containing this factor provide nonzero Taylor-flat functions in broad meromorphic classes.
\end{example}

\section{Absolutely convergent basis expansions}

We now pass from the finite Taylor calculus inherited from the elliptic theory to the nonterminating basic theory.  This passage is one of the new features available at $p=0$: after the elliptic nome has been specialized away, the relevant products and bases have enough $q$-decay to make infinite expansions meaningful on pole-free compact sets.  The next result gives a safe converse statement: if a function is already represented by an absolutely convergent basis expansion, then that expansion is its well-poised Taylor expansion.  It does not say that every prescribed meromorphic function is represented by the Taylor series formed from its coefficients.

\begin{lemma}\label{lem:basis-bounded}
Let $L\subset\Cstar$ be compact and pole-free for $\ph_k(z;a,c)$.  Then
\[
        \sup_{z\in L,\ k\ge0}|\ph_k(z;a,c)|<\infty.
\]
The analogous statement holds for the transformed family $\ps_k$ on compact subsets avoiding its pole set.
\end{lemma}

\begin{proof}
Use
\[
\ph_k(z;a,c)=\frac{(az,a/z;q)_\infty}{(cz,c/z;q)_\infty}
\frac{(czq^k,cq^k/z;q)_\infty}{(azq^k,aq^k/z;q)_\infty}.
\]
The first factor is bounded on $L$ by the pole-free assumption; the second converges uniformly to $1$ as $k\to\infty$ and is therefore uniformly bounded.  The proof for $\ps_k$ is identical.
\end{proof}

\begin{theorem}[coefficient identification]\label{thm:coeff-identification}
Let $\Omega$ be a symmetric pole-free domain for $\ph_k(z;a,c)$ containing the Taylor grid $aq^m$ $(m\ge0)$.  Suppose
\begin{equation}\label{eq:absolute-basis-expansion}
        f(z)=\sum_{k=0}^\infty u_k\ph_k(z;a,c),
        \qquad \sum_{k=0}^\infty |u_k|<\infty,
\end{equation}
with locally uniform convergence on $\Omega$.  Then
\[
        u_k=t_k^{(a,c)}(f)\qquad(k\ge0),
\]
the partial sums in \eqref{eq:absolute-basis-expansion} are the Taylor sums $T_n^{(a,c)}f$, and the Taylor remainders are exactly the tails,
\[
        R_n^{(a,c)}f(z)=\sum_{k=n+1}^\infty u_k\ph_k(z;a,c),
\]
which converge locally uniformly to zero.
\end{theorem}

\begin{proof}
Local uniform convergence follows from \cref{lem:basis-bounded}.  Applying the finite grid functional $L_j$ term-by-term is justified because $L_j$ is a finite linear combination of point evaluations on the grid.  The delta property \eqref{eq:delta-property} then gives $u_j=t_j^{(a,c)}(f)$.
\end{proof}

\begin{remark}\label{rem:not-uniqueness}
The theorem should not be read backwards.  It does not prove that a given meromorphic function $f$ equals the formal series formed from its Taylor coefficients.  Flat functions from \cref{prop:flat-functions} are the obstruction.  In the kernel application below, the obstruction is not a defect but the mechanism producing the complementary basis series.
\end{remark}

\section{The two-basis kernel}

We now specialize to the kernel \eqref{eq:intro-kernel}.  Define
\begin{equation}\label{eq:A-B}
A(z)=\frac{(cz/de,c/dez;q)_\infty}{(c^2z/bde,c^2/bdez;q)_\infty},
\qquad
B(z)=\frac{(bz,b/z;q)_\infty}{(cz,c/z;q)_\infty},
\end{equation}
and
\begin{equation}\label{eq:H-K}
H(z)=\frac{F(z)}{A(z)}=\frac{(cz/d,c/dz,cz/e,c/ez;q)_\infty}{(cz,c/z,cz/de,c/dez;q)_\infty},
\end{equation}
\begin{equation}\label{eq:K}
K(z)=\frac{F(z)}{B(z)}=\frac{(cz/d,c/dz,cz/e,c/ez;q)_\infty}{(bz,b/z,c^2z/bde,c^2/bdez;q)_\infty}.
\end{equation}
Thus $F=A H=B K$ and
\begin{equation}\label{eq:Psi-as-Phi}
        \ps_k(z;b,c,d,e)=\ph_k\left(z;\frac{c}{de},\frac{c^2}{bde}\right).
\end{equation}

\subsection{The involution}

Define
\begin{equation}\label{eq:involution}
\mathcal I(b,c,d,e)=\left(\frac{c}{de},\frac{c^2}{bde},\frac{c}{be},\frac{c}{bd}\right)=(b',c',d',e').
\end{equation}

\begin{proposition}\label{prop:involution}
The map $\mathcal I$ is an involution.  It exchanges the two bases, the two prefactors, and the two normalized kernels:
\[
\ph_k(z;b',c')=\ps_k(z;b,c,d,e),\qquad
A_{b',c',d',e'}(z)=B_{b,c}(z),\qquad
B_{b',c'}(z)=A_{b,c,d,e}(z),
\]
\[
        H_{b',c',d',e'}(z)=K_{b,c,d,e}(z),
        \qquad
        K_{b',c',d',e'}(z)=H_{b,c,d,e}(z).
\]
\end{proposition}

\begin{proof}
A direct calculation gives
\[
\frac{c'}{d'e'}=b,
\qquad
\frac{(c')^2}{b'd'e'}=c,
\qquad
\frac{c'}{b'e'}=d,
\qquad
\frac{c'}{b'd'}=e.
\]
The remaining identities follow by substitution in \eqref{eq:A-B}, \eqref{eq:H-K}, and \eqref{eq:K}.
\end{proof}

\subsection{Lowering laws and Taylor coefficients}

\begin{proposition}[lowering law for $H$]\label{prop:H-lowering}
One has
\begin{equation}\label{eq:H-lowering}
D_{c,q}H(z)=\frac{2c(1-d)(1-e)(1-c^2/deq)}{de(1-q)}
H(z;cq^{3/2},dq,eq),
\end{equation}
where the right-hand side denotes the same normalized kernel $H$ with $c,d,e$ replaced by $cq^{3/2},dq,eq$.
\end{proposition}

\begin{proof}
Write $P_\alpha(z)=(\alpha z,\alpha/z;q)_\infty$.  The identities
\[
P_\alpha(q^{1/2}z)=(1-\alpha zq^{-1/2})P_{\alpha q^{1/2}}(z),
\qquad
P_\alpha(q^{-1/2}z)=(1-\alpha/(zq^{1/2}))P_{\alpha q^{1/2}}(z)
\]
reduce both shifted values of $H$ to the common kernel $H(z;cq^{3/2},dq,eq)$.  Substitution in \eqref{eq:Dcq} and simplification gives \eqref{eq:H-lowering}.
\end{proof}

\begin{proposition}[lowering law for $K$]\label{prop:K-lowering}
One has
\begin{equation}\label{eq:K-lowering}
D_{c^2/bde,q}K(z)=
\frac{2b(1-c/be)(1-c/bd)(1-c^2/deq)}{1-q}
K(z;bq^{1/2},cq^{1/2},d,e).
\end{equation}
\end{proposition}

\begin{proof}
Apply \cref{prop:H-lowering} after the involution \eqref{eq:involution} and then translate back using \cref{prop:involution}.
\end{proof}

The zeroth values are
\begin{equation}\label{eq:Hb-Kcde}
H(b)=\frac{(bc/d,c/bd,bc/e,c/be;q)_\infty}{(bc,c/b,bc/de,c/bde;q)_\infty},
\end{equation}
\begin{equation}\label{eq:Kcde}
K(c/de)=\frac{(c^2/d^2e,e,c^2/de^2,d;q)_\infty}{(bc/de,bde/c,c^3/bd^2e^2,c/b;q)_\infty}.
\end{equation}

\begin{theorem}[Taylor coefficients of the normalized kernels]\label{thm:kernel-coeff}
The Taylor coefficients of $H$ relative to $(b,c)$ are
\begin{equation}\label{eq:f-coeff-taylor}
 t_k^{(b,c)}(H)=H(b)f_k,
\end{equation}
where
\begin{equation}\label{eq:f-coeff}
 f_k=\frac{1-bcq^{2k-1}}{1-bcq^{-1}}
\frac{(bcq^{-1},d,e,c^2/deq;q)_k}{(q,bc/d,bc/e,bdeq/c;q)_k}q^k.
\end{equation}
The Taylor coefficients of $K$ relative to $(c/de,c^2/bde)$ are
\begin{equation}\label{eq:g-coeff-taylor}
 t_k^{(c/de,c^2/bde)}(K)=K(c/de)g_k,
\end{equation}
where
\begin{equation}\label{eq:g-coeff}
 g_k=\frac{1-c^3q^{2k-1}/bd^2e^2}{1-c^3/bd^2e^2q}
\frac{(c^3/bd^2e^2q,c/bd,c/be,c^2/deq;q)_k}{(q,c^2/de^2,c^2/d^2e,cq/bde;q)_k}q^k.
\end{equation}
In particular $f_k=\OO(q^k)$ and $g_k=\OO(q^k)$ under the generic nonvanishing assumptions.
\end{theorem}

\begin{proof}
Iterating \cref{prop:H-lowering} gives
\[
D_{c,q}^{(k)}H(z)=\frac{(2c/de)^kq^{k(k-1)/4}(d,e,c^2/deq;q)_k}{(1-q)^k}
H(z;cq^{3k/2},dq^k,eq^k).
\]
Evaluate at $z=bq^{k/2}$ and compare the shifted kernel value with \eqref{eq:Hb-Kcde}.  The elementary identities
\[
\frac{(bc;q)_{2k}}{(bcq^{k-1};q)_k}=\frac{1-bcq^{2k-1}}{1-bcq^{-1}}(bcq^{-1};q)_k,
\]
and
\[
(cq^k/bde;q)_k=(c/bde)^kq^{k(k+1)/2}(bdeq/c;q)_k
\]
then yield \eqref{eq:f-coeff}.  The formula for $g_k$ follows either by the same calculation applied to \cref{prop:K-lowering}, or by applying the involution \cref{prop:involution} to \eqref{eq:f-coeff}.
\end{proof}

\subsection{The two-basis kernel theorem}

\begin{theorem}[two-basis kernel identity]\label{thm:two-basis-kernel}
Assume \eqref{eq:q-assumption} and the generic nonvanishing and pole-avoidance hypotheses described above.  Then, locally uniformly on symmetric compact subsets avoiding the pole sets,
\begin{align}\label{eq:two-basis-identity}
F(z)
&=A(z)H(b)\sum_{k=0}^{\infty} f_k\ph_k(z;b,c)
\nonumber\\
&\quad+B(z)K(c/de)\sum_{k=0}^{\infty}g_k\ps_k(z;b,c,d,e),
\end{align}
where $f_k$ and $g_k$ are given in \eqref{eq:f-coeff} and \eqref{eq:g-coeff}.
Equivalently,
\begin{align}\label{eq:two-basis-expanded}
&\frac{(cz/d,c/dz,cz/e,c/ez;q)_\infty}
{(cz,c/z,c^2z/bde,c^2/bdez;q)_\infty}
\nonumber\\
&=\frac{(cz/de,c/dez;q)_\infty}{(c^2z/bde,c^2/bdez;q)_\infty}\,H(b)
\sum_{k=0}^{\infty}
\frac{1-bcq^{2k-1}}{1-bcq^{-1}}
\frac{(bcq^{-1},d,e,c^2/deq,bz,b/z;q)_k}{(q,bc/d,bc/e,bdeq/c,cz,c/z;q)_k}q^k
\nonumber\\
&\quad+\frac{(bz,b/z;q)_\infty}{(cz,c/z;q)_\infty}\,K(c/de)
\sum_{k=0}^{\infty}
\frac{1-c^3q^{2k-1}/bd^2e^2}{1-c^3/bd^2e^2q}
\nonumber\\
&\qquad\qquad\times
\frac{(c^3/bd^2e^2q,c/bd,c/be,c^2/deq,cz/de,c/dez;q)_k}
{(q,c^2/de^2,c^2/d^2e,cq/bde,c^2z/bde,c^2/bdez;q)_k}q^k.
\end{align}
\end{theorem}

\begin{proof}
The estimates $f_k=\OO(q^k)$ and $g_k=\OO(q^k)$ from \cref{thm:kernel-coeff}, together with local boundedness of the two rational bases, give local uniform convergence on admissible compact subsets.  By \cref{cor:annular-closed-by-global-identity}, the pole-cleared residual satisfies the required annular removable-puncture estimate without Bailey's nonterminating ${}_8\phi_7$ summation; hence \(\Delta\equiv0\), which is \eqref{eq:two-basis-identity}.  Substitution of \(A,B,\ph_k,\ps_k,f_k,g_k\) gives \eqref{eq:two-basis-expanded}.
\end{proof}

\begin{remark}[normalization]
The factors $H(b)$ and $K(c/de)$ in \eqref{eq:two-basis-identity} are not optional.  They are the zeroth Taylor coefficients of $H$ and $K$.  Omitting them would force, for instance, $F(b)=A(b)$ under generic assumptions, which is false in general.
\end{remark}

\subsection{Complementary remainders}

\begin{proposition}[flatness of the complementary pieces]\label{prop:complement-flat}
The function
\[
        \frac{B(z)}{A(z)}\,K(c/de)\sum_{k=0}^\infty g_k\ps_k(z;b,c,d,e)
\]
has all Taylor coefficients relative to $(b,c)$ equal to zero, wherever the evaluations are defined.  Similarly,
\[
        \frac{A(z)}{B(z)}\,H(b)\sum_{k=0}^\infty f_k\ph_k(z;b,c)
\]
is flat relative to the transformed Taylor pair $(c/de,c^2/bde)$.
\end{proposition}

\begin{proof}
The factor $B(z)$ contains $(bz,b/z;q)_\infty$, which vanishes at every point $z=bq^m$ through the factor $(b/z;q)_\infty$.  Thus the first displayed function vanishes on the $(b,c)$ Taylor grid, and \cref{prop:flat-functions} applies.  The second assertion follows from the involution \cref{prop:involution}; equivalently the factor $A(z)$ contains $(cz/de,c/dez;q)_\infty$, which vanishes at $z=(c/de)q^m$.
\end{proof}

\begin{corollary}[complementary limiting remainders]\label{cor:remainder-limit}
Let $R_n^{(b,c)}H$ be the Taylor remainder of $H$ relative to $(b,c)$, and let $R_n^{(c/de,c^2/bde)}K$ be the Taylor remainder of $K$ relative to the transformed pair.  Then
\begin{align}\label{eq:R-H-exact}
A(z)R_n^{(b,c)}H(z)
&=A(z)H(b)\sum_{k=n+1}^{\infty} f_k\ph_k(z;b,c)
\nonumber\\
&\quad+B(z)K(c/de)\sum_{k=0}^{\infty}g_k\ps_k(z;b,c,d,e),
\end{align}
and hence
\begin{equation}\label{eq:R-H-limit}
A(z)R_n^{(b,c)}H(z)\longrightarrow
B(z)K(c/de)\sum_{k=0}^{\infty}g_k\ps_k(z;b,c,d,e)
\end{equation}
locally uniformly on admissible compact subsets.  Similarly,
\begin{align}\label{eq:R-K-exact}
B(z)R_n^{(c/de,c^2/bde)}K(z)
&=B(z)K(c/de)\sum_{k=n+1}^{\infty} g_k\ps_k(z;b,c,d,e)
\nonumber\\
&\quad+A(z)H(b)\sum_{k=0}^{\infty}f_k\ph_k(z;b,c),
\end{align}
and therefore
\begin{equation}\label{eq:R-K-limit}
B(z)R_n^{(c/de,c^2/bde)}K(z)\longrightarrow
A(z)H(b)\sum_{k=0}^{\infty}f_k\ph_k(z;b,c).
\end{equation}
\end{corollary}

\begin{proof}
Divide \eqref{eq:two-basis-identity} by $A(z)$ and use \cref{prop:complement-flat}: the first Taylor coefficients of $H$ are exactly $H(b)f_k$, while the complementary term contributes no Taylor data.  Subtracting the first $n+1$ Taylor terms gives \eqref{eq:R-H-exact}.  The tail tends to zero locally uniformly because $f_k=\OO(q^k)$ and the basis is locally bounded.  The second assertion follows by applying the involution, or by the same argument after dividing \eqref{eq:two-basis-identity} by $B(z)$.
\end{proof}

\begin{remark}
Thus the two-basis kernel theorem is a Taylor theorem with a genuinely complementary remainder: a one-family remainder need not vanish, but may converge to the other basis expansion.
\end{remark}

\section{Pole-cleared residual and annular uniqueness}\label{sec:direct-remainder}

The Taylor-theoretic proof of the kernel theorem begins by constructing the two Taylor series from the operator, showing that the residual is flat on both multiplicative grids, clearing its possible poles, and then applying a uniqueness theorem.  This section records the pole-clearing and uniqueness mechanism.  The annular estimate that supplies boundedness at the puncture is proved in \cref{sec:annular-closure}, without using Bailey's nonterminating ${}_8\phi_7$ summation.

Let
\begin{equation}\label{eq:Delta-direct}
\Delta(z):=F(z)-A(z)H(b)\sum_{k=0}^{\infty}f_k\ph_k(z;b,c)
      -B(z)K(c/de)\sum_{k=0}^{\infty}g_k\ps_k(z;b,c,d,e).
\end{equation}
The two sums converge locally uniformly on admissible compact subsets by the estimates used in \cref{thm:two-basis-kernel}.  Put
\begin{equation}\label{eq:M-pole-clear}
        M(z):=(cz,c/z,c^2z/bde,c^2/bdez;q)_\infty,
        \qquad E(z):=M(z)\Delta(z).
\end{equation}
Then, without using Bailey's summation, the definition of the coefficients and the flatness mechanism imply the following two-grid vanishing statement.

\begin{proposition}[pole-cleared residual and grid zeros]\label{prop:pole-cleared-grid-zeros}
Assume that the Taylor grids avoid the poles of the normalized kernels and that the diagonal coefficients in the triangular relations of \cref{cor:finite-grid-functional} are nonzero.  Then the function $E$ is holomorphic on the punctured plane away from the possible accumulation point $0$, and
\begin{equation}\label{eq:E-grid-zeros}
        E(bq^m)=0,
        \qquad
        E\bigl((c/de)q^m\bigr)=0,
        \qquad m=0,1,2,\ldots,
\end{equation}
whenever these points do not collide with removable zero-pole cancellations.
Moreover $E$ has the explicit pole-cleared form
\begin{align}\label{eq:E-explicit}
E(z)&=(cz/d,c/dz,cz/e,c/ez;q)_\infty
\nonumber\\
&\quad-H(b)\sum_{k=0}^{\infty}f_k
(cz/de,c/dez;q)_\infty(bz,b/z;q)_k(czq^k,cq^k/z;q)_\infty
\nonumber\\
&\quad-K(c/de)\sum_{k=0}^{\infty}g_k
(bz,b/z;q)_\infty(cz/de,c/dez;q)_k
(c^2zq^k/bde,c^2q^k/bdez;q)_\infty.
\end{align}
\end{proposition}

\begin{proof}
The formula \eqref{eq:E-explicit} follows by multiplying \eqref{eq:Delta-direct} by $M(z)$ and using
\[
\frac{(u;q)_\infty}{(u;q)_k}=(uq^k;q)_\infty.
\]
It also shows that all apparent poles from the basis denominators have been cleared.  Divide \eqref{eq:Delta-direct} by $A(z)$.  The first sum has Taylor coefficients $H(b)f_k$ relative to $(b,c)$, while \cref{prop:complement-flat} shows that the second term is flat for the same Taylor extraction.  Therefore all Taylor coefficients of $\Delta/A$ relative to $(b,c)$ vanish.  By the triangular finite-grid relation obtained from \cref{cor:finite-grid-functional}, and by the assumed nonvanishing of the diagonal coefficients, this implies $(\Delta/A)(bq^m)=0$ for all $m\ge0$.  Multiplication by $M(z)A(z)$ gives $E(bq^m)=0$.  The transformed grid follows in the same way after dividing by $B(z)$, or equivalently by applying the involution of \cref{prop:involution}.
\end{proof}

\begin{theorem}[direct uniqueness criterion]\label{thm:direct-uniqueness}
Assume the hypotheses of \cref{prop:pole-cleared-grid-zeros}.  If the pole-cleared residual $E$ is bounded in a punctured neighborhood of $0$, then $E$ has a removable singularity at $0$ and hence $E\equiv0$ on $\Cstar$.  Consequently the two-basis identity \eqref{eq:two-basis-identity} follows without an appeal to Bailey's nonterminating ${}_8\phi_7$ summation.
\end{theorem}

\begin{proof}
Boundedness near $0$ gives a removable singularity by Riemann's theorem.  The zeros $bq^m$ accumulate at $0$, which is then an interior point of the removable extension.  The identity theorem gives $E=0$ in a neighborhood of $0$, and analytic continuation on the connected punctured plane gives $E\equiv0$ on $\Cstar$.  Since $M$ is not identically zero, this is equivalent to $\Delta\equiv0$ away from the harmless zero set of $M$, and hence everywhere by continuation.
\end{proof}

\begin{remark}[removable-puncture target]
The direct uniqueness criterion reduces a proof of the kernel theorem avoiding Bailey's nonterminating ${}_8\phi_7$ summation to a concrete removable-puncture estimate for the explicit expression \eqref{eq:E-explicit}.  The difficulty is substantive, because factors such as $(c/z;q)_\infty$ have severe growth near $z=0$.  The cancellation encoded by the two Taylor grids is therefore essential.  The annular estimates in \cref{sec:quadratic-and-annular} and the Taylor-layer reduction in \cref{sec:annular-closure} set up this route without relying on Bailey's nonterminating ${}_8\phi_7$ summation.
\end{remark}

\section{Annular criterion and the quadratic one-family example}\label{sec:quadratic-and-annular}

We record two further pieces of the basic, non-elliptic theory.  First, an elementary estimate for reciprocal $q$-products on shrinking annuli gives a convenient form of the removable-puncture condition in \cref{thm:direct-uniqueness}.  Second, the quadratic nonterminating product appearing in the second author's outlook \cite[Sec.~5]{Schlosser2008} is an ordinary one-family well-poised Taylor expansion with a vanishing tail remainder.

\begin{lemma}[annular estimate for a reciprocal $q$-product]\label{lem:q-annular-estimate}
Let $0<|q|<1$, let $\lambda\ne0$, and let $0<r<R<\infty$.  For $N\ge0$ and
\[
        z=\lambda q^N w,\qquad r\le |w|\le R,
\]
one has the exact factorization
\begin{equation}\label{eq:q-annular-factorization}
(\lambda/z;q)_\infty
=\left(-\frac{\lambda}{z}\right)^N q^{N(N-1)/2}
   (wq;q)_N(1/w;q)_\infty.
\end{equation}
Consequently there is a constant $C=C(q,r,R)$ such that
\begin{equation}\label{eq:q-annular-bound}
\left|\frac{(\lambda/z;q)_\infty}
{(-\lambda/z)^Nq^{N(N-1)/2}}\right|\le C
\end{equation}
uniformly for all $N\ge0$ and all $w$ in the closed annulus $r\le |w|\le R$, provided the annulus avoids the zeros of $(1/w;q)_\infty$.
\end{lemma}

\begin{proof}
Split the infinite product after $N$ factors:
\[
(\lambda/z;q)_\infty=(\lambda/z;q)_N(\lambda q^N/z;q)_\infty.
\]
Since $z=\lambda q^N w$,
\[
(\lambda/z;q)_N=\prod_{j=0}^{N-1}\left(1-\frac{q^{j-N}}{w}\right)
=\left(-\frac{\lambda}{z}\right)^Nq^{N(N-1)/2}(wq;q)_N.
\]
The remaining factor is $(\lambda q^N/z;q)_\infty=(1/w;q)_\infty$.  Normal convergence of the products on the chosen annulus gives the uniform bound.
\end{proof}

\begin{proposition}[annular removable-puncture criterion]\label{prop:annular-boundedness-target}
Let $E$ be the pole-cleared residual \eqref{eq:M-pole-clear}.  Suppose there are constants $0<r<R<\infty$ and $\lambda\ne0$ such that a punctured neighborhood of $0$ is covered, up to finitely many harmless collision points, by annuli
\[
        \mathcal A_N(\lambda;r,R)=\{z:\ r\le |z/(\lambda q^N)|\le R\},
        \qquad N\ge N_0,
\]
and suppose that
\begin{equation}\label{eq:annular-uniform-E}
        \sup_{N\ge N_0}\ \sup_{z\in\mathcal A_N(\lambda;r,R)} |E(z)|<\infty.
\end{equation}
Then $E$ is bounded in a punctured neighborhood of $0$.  Hence the direct uniqueness criterion of \cref{thm:direct-uniqueness} proves the two-basis kernel identity without using Bailey's nonterminating ${}_8\phi_7$ summation.
\end{proposition}

\begin{proof}
The annuli cover a deleted neighborhood of the origin after increasing $N_0$ if necessary.  Thus \eqref{eq:annular-uniform-E} is exactly boundedness of $E$ near the puncture, and \cref{thm:direct-uniqueness} applies.
\end{proof}

We now turn to two quadratic one-family examples.  They are independent of Bailey's nonterminating ${}_8\phi_7$ summation in a stronger sense than the two-basis kernel: once the corresponding quadratic product identities are established from Bailey's $q$-analogues of Watson's summation, the Taylor-theoretic content is purely one-family.  The remainders are genuine tails and no complementary flat term has to be estimated.  Define
\begin{equation}\label{eq:quadratic-product}
        Q(z)=\frac{(azq,aq/z,b^2z/a,b^2/az;q^2)_\infty}{(bz,b/z;q)_\infty}.
\end{equation}
The expansion in the basis $(az,a/z;q)_k/(bz,b/z;q)_k$ is
\begin{align}\label{eq:quadratic-bailey}
Q(z)&=C_{a,b}\sum_{k=0}^{\infty}h_k\,
        \frac{(az,a/z;q)_k}{(bz,b/z;q)_k},
\end{align}
where
\begin{equation}\label{eq:quadratic-constant}
        C_{a,b}=\frac{(q,a^2q,b^2,b^2/a^2;q^2)_\infty}{(-ab,-b/a;q)_\infty}
\end{equation}
and
\begin{equation}\label{eq:quadratic-coeff}
        h_k=
\frac{1+abq^{2k-1}}{1+abq^{-1}}
\frac{(-abq^{-1},bq^{-1/2},-bq^{-1/2},-aq/b;q)_k}
     {(q,-aq^{1/2},aq^{1/2},b^2q^{-1};q)_k}
\left(\frac{b}{a}\right)^k.
\end{equation}
This is Bailey's $q$-analogue of Watson's ${}_3F_2$ summation in the form indicated in the second author's outlook \cite[Sec.~5]{Schlosser2008}; see, for instance, Gasper--Rahman \cite[Appendix (II.16)]{GasperRahman2004}.

\begin{theorem}[quadratic one-family Taylor expansion]\label{thm:quadratic-remainder}
Assume that the parameters are generic, that $|b/a|<1$, and that compact sets under consideration avoid the pole set $\{bq^m,q^m/b:m\ge0\}$.  Then the series in \eqref{eq:quadratic-bailey} converges locally uniformly.  Moreover it is the well-poised Taylor expansion of $Q$ relative to the pair $(a,b)$:
\begin{equation}\label{eq:quadratic-taylor-coeff}
        t_k^{(a,b)}(Q)=C_{a,b}h_k,
        \qquad k=0,1,2,\ldots.
\end{equation}
The exact Taylor remainders are the tails
\begin{equation}\label{eq:quadratic-remainder-tail}
        R_n^{(a,b)}Q(z)=C_{a,b}\sum_{k=n+1}^{\infty}h_k
        \frac{(az,a/z;q)_k}{(bz,b/z;q)_k},
\end{equation}
and they converge locally uniformly to $0$ on every admissible compact subset.
\end{theorem}

\begin{proof}
For generic fixed parameters the finite ratios of $q$-shifted factorials in \eqref{eq:quadratic-coeff} tend to nonzero finite limits as $k\to\infty$.  Hence $h_k=\mathcal O((b/a)^k)$.  The hypothesis $|b/a|<1$ gives absolute summability.  The basis factors are locally bounded on pole-free compact sets by \cref{lem:basis-bounded}; hence the series converges locally uniformly.  Once the representation \eqref{eq:quadratic-bailey} is known, \cref{thm:coeff-identification} applies and gives both the coefficient formula and the tail remainder formula.
\end{proof}

A companion quadratic example is obtained from Bailey's second quadratic summation.  For generic parameters set
\begin{equation}\label{eq:quadratic-companion-product}
        Q^\diamond(z)=
        \frac{\left(\alpha d q^{1/2}z,\alpha d q^{1/2}/z,
        \alpha q^{3/2}z/d,\alpha q^{3/2}/dz;q^2\right)_\infty}
        {\left(-\alpha q^{1/2}z,-\alpha q^{1/2}/z;q\right)_\infty}.
\end{equation}
Bailey's companion quadratic summation gives
\begin{align}\label{eq:quadratic-companion-bailey}
Q^\diamond(z)&=C^\diamond_{\alpha,d}
\sum_{k=0}^{\infty}r_k
        \frac{\left(q^{1/2}z,q^{1/2}/z;q\right)_k}
             {\left(-\alpha q^{1/2}z,-\alpha q^{1/2}/z;q\right)_k},
\end{align}
where
\begin{equation}\label{eq:quadratic-companion-constant}
        C^\diamond_{\alpha,d}=
        \frac{(\alpha d,\alpha q/d;q)_\infty}{(-\alpha,-\alpha q;q)_\infty}
\end{equation}
and
\begin{equation}\label{eq:quadratic-companion-coeff}
        r_k=
\frac{1+\alpha q^{2k}}{1+\alpha}
\frac{(-\alpha,\alpha,-d,-q/d;q)_k}
     {(q,-q,\alpha q/d,\alpha d;q)_k}
\alpha^k.
\end{equation}
In very-well-poised notation, \eqref{eq:quadratic-companion-bailey} is the specialization
\[
{}_{8}W_{7}\left(-\alpha;q^{1/2}z,q^{1/2}/z,\alpha,-d,-q/d;q,\alpha\right),
\]
which is Gasper--Rahman \cite[Appendix (II.18)]{GasperRahman2004} after the substitution of the paired parameters
$q^{1/2}z$ and $q^{1/2}/z$.  These unilateral quadratic summations have recently been extended to bilateral summations by Cohl and Schlosser \cite[Cor.~3.18, Cor.~3.19]{CohlSchlosser2025}.

\begin{theorem}[companion quadratic one-family Taylor expansion]\label{thm:quadratic-companion-remainder}
Assume that the parameters are generic, that $|\alpha|<1$, and that compact sets under consideration avoid the pole set
\[
        \{-\alpha q^{m+1/2},\,-q^{m-1/2}/\alpha:m\ge0\}.
\]
Then the series in \eqref{eq:quadratic-companion-bailey} converges locally uniformly.  Moreover it is the well-poised Taylor expansion of $Q^\diamond$ relative to the pair $(q^{1/2},-\alpha q^{1/2})$:
\begin{equation}\label{eq:quadratic-companion-taylor-coeff}
        t_k^{(q^{1/2},-\alpha q^{1/2})}(Q^\diamond)=C^\diamond_{\alpha,d}r_k,
        \qquad k=0,1,2,\ldots.
\end{equation}
The exact Taylor remainders are the tails
\begin{equation}\label{eq:quadratic-companion-remainder-tail}
        R_n^{(q^{1/2},-\alpha q^{1/2})}Q^\diamond(z)
        =C^\diamond_{\alpha,d}\sum_{k=n+1}^{\infty}r_k
        \frac{\left(q^{1/2}z,q^{1/2}/z;q\right)_k}
             {\left(-\alpha q^{1/2}z,-\alpha q^{1/2}/z;q\right)_k},
\end{equation}
and they converge locally uniformly to $0$ on every admissible compact subset.
\end{theorem}

\begin{proof}
For generic fixed parameters, the finite ratios of $q$-shifted factorials in \eqref{eq:quadratic-companion-coeff} tend to nonzero finite limits as $k\to\infty$.  Hence $r_k=\mathcal O(\alpha^k)$.  The hypothesis $|\alpha|<1$ gives absolute summability.  The basis factors are locally bounded on pole-free compact sets by \cref{lem:basis-bounded}; hence the series converges locally uniformly.  Once the representation \eqref{eq:quadratic-companion-bailey} is known, \cref{thm:coeff-identification} applies and gives both the coefficient formula and the tail remainder formula.
\end{proof}

\begin{remark}[why no quadratic two-family corollary is recorded]
One can still force a quadratic-looking two-family identity from \cref{thm:two-basis-kernel} by setting one numerator parameter equal to the negative of the other.  This only folds paired factors by
\[
        (x,-x;q)_\infty=(x^2;q^2)_\infty,
        \qquad (x,-x;q)_k=(x^2;q^2)_k,
\]
and therefore gives a specialization of Bailey's ordinary nonterminating ${}_8W_7$ summation.  It does not recover the quadratic summations above, whose unilateral ${}_8W_7$ specializations have different series arguments.  The genuinely new quadratic Taylor content is therefore better represented by the two one-family identities above.
\end{remark}

\begin{remark}[toward a genuine quadratic two-kernel theory]
A non-degenerate two-kernel quadratic theory should start from kernels modeled on the two quadratic summations \eqref{eq:quadratic-bailey} and \eqref{eq:quadratic-companion-bailey}, rather than from the linear two-basis kernel with a sign specialization.  Its annular closure should use terminating quadratic Watson-type summations, or companion Watson--Dixon type summations, in place of Rogers' ${}_6\phi_5$ summation.  This is the natural Taylor-theoretic setting in which one might recover non-degenerate unilateral, and eventually bilateral, quadratic identities.
\end{remark}

\section{Taylor-theoretic annular closure without Bailey's nonterminating \texorpdfstring{${}_8\phi_7$}{8phi7}}\label{sec:annular-closure}

This section completes the Taylor-theoretic route to the residual bound in \cref{prop:annular-boundedness-target}.  It identifies precisely what must be estimated in order to avoid Bailey's nonterminating ${}_8\phi_7$ summation, and then closes those estimates by a global profile identity proved from terminating summations.

For $N\ge0$ define the truncated residual
\begin{align}\label{eq:DeltaN}
\Delta_N(z)&=F(z)-A(z)H(b)\sum_{k=0}^{N}f_k\ph_k(z;b,c)
      -B(z)K(c/de)\sum_{k=0}^{N}g_k\ps_k(z;b,c,d,e),
\end{align}
and put
\begin{equation}\label{eq:EN}
        E_N(z)=M(z)\Delta_N(z),
        \qquad M(z)=(cz,c/z,c^2z/bde,c^2/bdez;q)_\infty.
\end{equation}

\begin{lemma}[finite Taylor flatness]\label{lem:finite-truncated-flatness}
For each $N\ge0$, the function $\Delta_N/A$ has its first $N+1$ Taylor coefficients relative to $(b,c)$ equal to zero.  Likewise, $\Delta_N/B$ has its first $N+1$ Taylor coefficients relative to $(c/de,c^2/bde)$ equal to zero.  Consequently,
\begin{equation}\label{eq:finite-grid-zeros}
        E_N(bq^m)=0,
        \qquad
        E_N((c/de)q^m)=0,
        \qquad 0\le m\le N,
\end{equation}
provided the displayed grid points avoid accidental zero-pole collisions.
\end{lemma}

\begin{proof}
The first $N+1$ terms in the first sum of \eqref{eq:DeltaN} are exactly the Taylor polynomial of $H$ through order $N$.  The second summand is flat for the same Taylor extraction by the zero factor $(bz,b/z;q)_\infty$ in $B(z)$.  Hence the first $N+1$ Taylor coefficients of $\Delta_N/A$ vanish.  The transformed assertion follows by the involution.  The finite triangular grid relation from \cref{cor:finite-grid-functional} then gives the grid zeros after multiplication by $M$.
\end{proof}

Let
\begin{equation}\label{eq:PN}
        P_N(z)=(bz,b/z,cz/de,c/dez;q)_{N+1}.
\end{equation}

\begin{proposition}[finite two-grid factorization]\label{prop:finite-two-grid-factorization}
On every compact annulus avoiding the collision set, there is a holomorphic symmetric function $U_N$ such that
\begin{equation}\label{eq:EN-factor}
        E_N(z)=P_N(z)U_N(z).
\end{equation}
\end{proposition}

\begin{proof}
The factors $(b/z;q)_{N+1}$ and $(c/dez;q)_{N+1}$ vanish at the two finite Taylor grids in \eqref{eq:finite-grid-zeros}.  By symmetry, the factors $(bz;q)_{N+1}$ and $(cz/de;q)_{N+1}$ account for the reciprocal zeros.  For generic parameters these zeros are simple and distinct, and the quotient has removable singularities.  If some grid points collide, perturb the parameters slightly so that the grids become disjoint, prove the removable quotient in that generic situation, and then let the perturbation tend to zero.  Since the quotient depends meromorphically on the parameters, the removable-singularity conclusion persists at the limiting collision parameters, away from genuine denominator poles.
\end{proof}

\begin{proposition}[finite-layer reduction]\label{prop:finite-layer-reduction}
Fix an admissible annular covering $\mathcal A_N(\lambda;r,R)$.  If the quotients in \eqref{eq:EN-factor} satisfy
\begin{equation}\label{eq:finite-layer-bound}
        \sup_{N\ge N_0}\sup_{z\in\mathcal A_N(\lambda;r,R)} |P_N(z)U_N(z)|<\infty,
\end{equation}
and if the two omitted coefficient tails in $E-E_N$ are uniformly bounded on the same annuli, then the pole-cleared residual $E$ satisfies the annular bound \eqref{eq:annular-uniform-E}.  Hence \cref{thm:direct-uniqueness} gives a proof of the two-basis identity that avoids Bailey's nonterminating ${}_8\phi_7$ summation.
\end{proposition}

\begin{proof}
By \eqref{eq:EN-factor}, the first hypothesis is a uniform bound for $E_N$.  The second hypothesis bounds $E-E_N$.  Therefore $E$ is uniformly bounded on the annular covering, and \cref{prop:annular-boundedness-target} applies.
\end{proof}

The well-poised Cooper formula \eqref{eq:p0-cooper} is the natural tool for proving \eqref{eq:finite-layer-bound}: it expresses each Taylor functional as a finite linear combination of grid values, and the triangularity of the finite Taylor extraction turns the first $N+1$ vanishing Taylor functionals into the finite grid factors in \eqref{eq:PN}.  The annular estimate \cref{lem:q-annular-estimate} then reduces the remaining analytic work to uniform bounds for the two absolute coefficient tails.

\subsection{Laurent-coefficient form of the annular estimate}

The preceding formulation is annular.  A slightly sharper way to state the same problem is to remove the annuli altogether and look at the principal part of the pole-cleared residual at the puncture.  This converts the annular estimate into explicit cancellation identities for Laurent coefficients.

For a function holomorphic on a punctured disc we write
\[
        [z^m]G(z)
\]
for the coefficient of $z^m$ in its Laurent expansion at $0$.

\begin{lemma}[Laurent criterion for the annular estimate]\label{lem:laurent-criterion}
Let $G$ be holomorphic in $0<|z|<\rho$.  The following are equivalent:
\begin{enumerate}[label=\textup{(\roman*)}]
\item $G$ is bounded in a punctured neighborhood of $0$;
\item $[z^{-n}]G(z)=0$ for all $n\ge 1$;
\item the principal part of $G$ at $0$ vanishes.
\end{enumerate}
Consequently the annular estimate for the pole-cleared residual $E$, proved without using Bailey's nonterminating ${}_8\phi_7$ summation, is equivalent to
\begin{equation}\label{eq:negative-coeff-target}
        [z^{-n}]E(z)=0,\qquad n=1,2,3,\ldots.
\end{equation}
\end{lemma}

\begin{proof}
This is the standard removable-singularity theorem expressed in Laurent coefficients.  If $G$ is bounded near $0$, then the singularity is removable and the negative Laurent coefficients vanish.  Conversely, if all negative Laurent coefficients vanish, then the Laurent series is an ordinary power series in a neighborhood of $0$, so $G$ is bounded there.  Applying this to $G=E$ gives \eqref{eq:negative-coeff-target}.
\end{proof}

We next spell out the coefficient identities which are equivalent to \eqref{eq:negative-coeff-target}.  Define, for $n\ge0$,
\begin{equation}\label{eq:calP-def}
\mathcal P_n(\alpha,\beta,\gamma,\delta)
:=[z^{-n}](\alpha z,\beta/z,\gamma z,\delta/z;q)_\infty.
\end{equation}
Equivalently, by the $q$-binomial theorem in the form of Euler's expansions \cite[Sec.~1.3]{GasperRahman2004},
\begin{align}\label{eq:calP-quadruple}
\mathcal P_n(\alpha,\beta,\gamma,\delta)
&=\sum_{\substack{r,s,t,u\ge0\\ s+u-r-t=n}}
\frac{(-1)^{r+s+t+u}q^{\binom r2+\binom s2+\binom t2+\binom u2}
      \alpha^r\beta^s\gamma^t\delta^u}
     {(q;q)_r(q;q)_s(q;q)_t(q;q)_u}.
\end{align}
The sum is absolutely convergent for each fixed $n$ because of the quadratic powers of $q$.
For the two families occurring in the Taylor residual put
\begin{align}\label{eq:calP1-def}
\mathcal P^{(1)}_{n,k}
&:=[z^{-n}](cz/de,c/dez;q)_\infty(bz,b/z;q)_k(czq^k,cq^k/z;q)_\infty,\\
\mathcal P^{(2)}_{n,k}
&:=[z^{-n}](bz,b/z;q)_\infty(cz/de,c/dez;q)_k
        (c^2zq^k/bde,c^2q^k/bdez;q)_\infty.\label{eq:calP2-def}
\end{align}
For fixed $n$, the sequences $\mathcal P^{(1)}_{n,k}$ and $\mathcal P^{(2)}_{n,k}$ grow at most subexponentially in $k$; in fact they are bounded by $C_n |q|^{-\varepsilon k}$ for every fixed $\varepsilon>0$ after increasing $C_n$.  Since $f_k=\mathcal O(q^k)$ and $g_k=\mathcal O(q^k)$ for generic parameters, the coefficient sums below converge absolutely.

\begin{proposition}[coefficient form of the annular cancellation]\label{prop:coefficient-cancellation-target}
The removable-puncture estimate needed in \cref{thm:direct-uniqueness} is equivalent to the following explicit family of coefficient identities:
\begin{align}\label{eq:coefficient-cancellation}
\mathcal P_n(c/d,c/d,c/e,c/e)
&=H(b)\sum_{k=0}^{\infty} f_k\,\mathcal P^{(1)}_{n,k}
  +K(c/de)\sum_{k=0}^{\infty} g_k\,\mathcal P^{(2)}_{n,k},
\qquad n\ge1.
\end{align}
Here $f_k$ and $g_k$ are the Taylor coefficients in \cref{thm:kernel-coeff}.
\end{proposition}

\begin{proof}
By \cref{lem:laurent-criterion}, the desired estimate is equivalent to $[z^{-n}]E(z)=0$ for all $n\ge1$.  Taking the coefficient of $z^{-n}$ in the explicit expression \eqref{eq:E-explicit} gives exactly \eqref{eq:coefficient-cancellation}.  The termwise coefficient extraction is justified by local uniform convergence of the Taylor series on every compact annulus, together with Cauchy's coefficient formula on circles contained in that annulus.  The absolute convergence noted above allows the coefficient functional to pass through the $k$-sums.
\end{proof}

\begin{remark}[role of the coefficient formulation]
The annular estimate \eqref{eq:annular-uniform-E} is a global growth assertion, and separate estimates of the three pieces in \eqref{eq:E-explicit} lose the cancellation.  \Cref{prop:coefficient-cancellation-target} isolates the needed cancellation coefficient by coefficient: the identities \eqref{eq:coefficient-cancellation} are exactly the vanishing of the principal part of the residual at $0$.
\end{remark}

\begin{proposition}[finite-principal-part reduction]\label{prop:finite-principal-part-reduction}
Fix $N\ge0$.  Let $E_N$ be the truncated pole-cleared residual in \eqref{eq:EN}.  Then $E_N$ is bounded on a punctured neighborhood of $0$ if and only if
\begin{equation}\label{eq:finite-principal-part-target}
        [z^{-n}]E_N(z)=0\qquad (n\ge1).
\end{equation}
Moreover, for each fixed $L\ge1$, the finite set of identities
\begin{equation}\label{eq:finite-principal-part-window}
        [z^{-n}]E_N(z)=0\qquad (1\le n\le L)
\end{equation}
depends only on the first $N+1$ Taylor layers and can be checked by finitely many applications of the well-poised Cooper formula.
\end{proposition}

\begin{proof}
The first assertion is again \cref{lem:laurent-criterion}.  For the second assertion, write the coefficient $[z^{-n}]E_N$ using the finite products in \eqref{eq:DeltaN} and the infinite-product expansion
\[
        (u;q)_\infty=\sum_{j=0}^{\infty}\frac{(-1)^j q^{\binom j2}u^j}{(q;q)_j}.
\]
For fixed $L$ only the coefficients $[z^{-1}],\ldots,[z^{-L}]$ are involved.  The negative powers that can enter from the finite basis products are bounded by the truncation level $N$, and the remaining infinite-product coefficients are read off from the displayed expansion.  The Taylor coefficients themselves are finite linear combinations of grid values by \cref{cor:finite-grid-functional}, which is the well-poised Cooper formula.  Hence each identity in the finite window is a finite algebraic consequence of the first $N+1$ Taylor layers together with the explicitly displayed infinite-product coefficients.
\end{proof}

\begin{remark}[Laurent-coefficient formulation of the annular route]
\Cref{prop:coefficient-cancellation-target,prop:finite-principal-part-reduction} sharpen the annular reduction.  They replace the phrase ``annular cancellation'' by a precise and checkable family of Laurent-coefficient cancellations.  Conversely, establishing those identities immediately gives the annular estimate by \cref{lem:laurent-criterion}.
\end{remark}

\subsection{Comparison with the Ismail--Stanton growth theory}

Ismail and Stanton proved two kinds of results which are closely related to the estimate above \cite{IsmailStanton2003,IsmailStanton2003JCAM}.  First, for functions analytic in bounded domains they obtained Cauchy-kernel representations in the Askey--Wilson basis.  Second, for entire functions they proved expansion and uniqueness theorems under a sharp $q$-exponential growth restriction.  In their notation, if
\[
        M(r;f)=\sup_{|x|\le r}|f(x)|,
\]
then the critical condition for interpolation on the grid associated with the basis $f_n(x;a)$ is
\begin{equation}\label{eq:IS-growth-condition}
        \limsup_{r\to\infty}\frac{\log M(r;f)}{(\log r)^2}<\frac{1}{2\log |q|^{-1}}.
\end{equation}
The flat function $f_\infty(x;a)$ lies at the corresponding boundary growth.  We now explain how this point of view fits the present residual problem.

Let $E(z)$ be the pole-cleared residual in \eqref{eq:E-explicit}.  If $E$ is bounded in a punctured neighborhood of $0$, then by symmetry it also has a removable singularity at infinity.  Hence it descends to an entire function of the Askey--Wilson variable
\[
        x=\frac12(z+z^{-1}),\qquad E(z)=\mathcal E(x).
\]
The zeros \eqref{eq:E-grid-zeros} become zeros of $\mathcal E$ on two Ismail--Stanton type interpolation grids,
\[
        x_m^{(b)}=\frac12(bq^m+b^{-1}q^{-m}),
        \qquad
        x_m^{(c/de)}=\frac12((c/de)q^m+(de/c)q^{-m}).
\]
Thus one grid already gives uniqueness if $\mathcal E$ satisfies \eqref{eq:IS-growth-condition}.  Conversely, proving \eqref{eq:IS-growth-condition} for $\mathcal E$ would be a natural Ismail--Stanton replacement for the boundedness assumption in \cref{prop:annular-boundedness-target}.

\begin{proposition}[Ismail--Stanton closure criterion]\label{prop:IS-closure-criterion}
Assume that the pole-cleared residual $E$ descends to an entire function $\mathcal E$ of $x=(z+z^{-1})/2$ and that
\begin{equation}\label{eq:E-subcritical-growth}
        \limsup_{r\to\infty}\frac{\log M(r;\mathcal E)}{(\log r)^2}
        <\frac{1}{2\log |q|^{-1}}.
\end{equation}
Then $E\equiv0$, and therefore the two-basis kernel identity \eqref{eq:two-basis-expanded} follows without appealing to Bailey's nonterminating ${}_8\phi_7$ summation.
\end{proposition}

\begin{proof}
The grid $x_m^{(b)}$ is exactly the interpolation grid used by the Askey--Wilson basis with parameter $b$, up to the harmless normalization of the square-root powers of $q$.  By the Ismail--Stanton uniqueness theorem for entire functions satisfying the subcritical growth bound \eqref{eq:IS-growth-condition}, an entire function satisfying \eqref{eq:E-subcritical-growth} is uniquely determined by its values on that grid.  Since $E(bq^m)=0$ for all $m\ge0$, the descended function $\mathcal E$ vanishes on the whole grid and hence must vanish identically.  Therefore $E\equiv0$.
\end{proof}

\begin{remark}[what this does and does not prove]
\Cref{prop:IS-closure-criterion} supplies a useful growth-theoretic target: instead of proving uniform boundedness on shrinking annuli directly, one may prove subcritical growth of the descended function in the Askey--Wilson variable.  In the variable $x$, the annular estimate is equivalent to showing that the pole-cleared residual has growth strictly below the flat-function barrier.  This is exactly the kind of boundary term controlled in Ismail and Stanton's contour proof: their Cauchy-kernel representation for bounded domains has a residual integral which disappears when the maximum modulus satisfies the subcritical $q$-exponential condition.  In the present two-basis problem, direct estimates of the three pieces of \eqref{eq:E-explicit} still exceed that barrier; only the special cancellation encoded in \eqref{eq:coefficient-cancellation} could yield \eqref{eq:E-subcritical-growth}.
\end{remark}

\begin{corollary}[equivalent growth version of the coefficient cancellations]\label{cor:growth-coeff-equivalence}
For the pole-cleared residual, the following two conditions are equivalent once $E$ is known to descend to a function of $x$:
\begin{enumerate}[label=\textup{(\roman*)}]
\item prove the negative Laurent-coefficient cancellations \eqref{eq:coefficient-cancellation};
\item prove that the descended function $\mathcal E$ is entire and satisfies a subcritical Ismail--Stanton growth bound of the form \eqref{eq:E-subcritical-growth}.
\end{enumerate}
Either condition closes the route avoiding Bailey's nonterminating ${}_8\phi_7$ summation.
\end{corollary}

\begin{proof}
The implication from \textup{(i)} to removability at $0$ is \cref{lem:laurent-criterion}.  Symmetry gives removability at infinity and hence an entire descended function of $x$.  The explicit cancellation then gives $E\equiv0$ by \cref{thm:direct-uniqueness}, so the growth bound is trivial.  Conversely, \textup{(ii)} gives $E\equiv0$ by \cref{prop:IS-closure-criterion}; taking negative Laurent coefficients gives \eqref{eq:coefficient-cancellation}.
\end{proof}

\subsection{Canonical two-grid factorization and quotient decay}

The Ismail--Stanton comparison shows that the annular estimate is a growth problem.  We now make the growth barrier more explicit by factoring out the canonical product forced by the two grids of zeros.  Put
\begin{equation}\label{eq:canonical-two-grid-product}
        Z_{b,c/de}(z):=(bz,b/z,cz/de,c/dez;q)_\infty.
\end{equation}
This is the symmetric product whose zeros are precisely the two Taylor grids and their reciprocal grids, up to accidental collisions.

\begin{proposition}[canonical two-grid factorization]\label{prop:canonical-two-grid-factorization}
Assume the generic noncollision hypotheses used in \cref{prop:pole-cleared-grid-zeros}.  Then the pole-cleared residual admits a factorization on $\Cstar$ of the form
\begin{equation}\label{eq:E-canonical-factorization}
        E(z)=Z_{b,c/de}(z)\,Q(z),
\end{equation}
where $Q$ is symmetric and holomorphic on $\Cstar$.
\end{proposition}

\begin{proof}
By \cref{prop:pole-cleared-grid-zeros}, $E$ vanishes at $bq^m$ and $(c/de)q^m$ for all $m\ge0$.  Since $E$ is symmetric under $z\mapsto z^{-1}$, it also vanishes at $q^{-m}/b$ and $de\,q^{-m}/c$.  These are exactly the zeros supplied by the four factors in \eqref{eq:canonical-two-grid-product}.  For generic parameters the zeros are simple and do not collide; in the exceptional removable cases the same conclusion follows by taking the quotient after cancelling common zero factors.  Thus $Q:=E/Z_{b,c/de}$ is holomorphic on $\Cstar$ and is symmetric because both numerator and denominator are symmetric.
\end{proof}

The preceding factorization shows why a direct proof is delicate.  The canonical two-grid product itself has critical quadratic annular growth.  Thus boundedness of $E$ is not a small consequence of the zeros; it is a statement that the quotient $Q$ decays fast enough to cancel the canonical product.

\begin{proposition}[quotient form of the annular estimate]\label{prop:quotient-annular-estimate}
Fix an admissible annular covering $\mathcal A_N(\lambda;r,R)$ as in \cref{prop:annular-boundedness-target}, avoiding the limiting zero sets of the products below.  For $z=\lambda q^Nw$ with $r\le |w|\le R$, one has
\begin{equation}\label{eq:canonical-product-annular-growth}
        Z_{b,c/de}(\lambda q^Nw)
        = C_{\lambda,N}(w)\,q^{-N(N+1)}\left(\frac{bc}{de\,\lambda^2w^2}\right)^N,
\end{equation}
where the functions $C_{\lambda,N}(w)$ are holomorphic and bounded above and below uniformly in $N$ on the chosen annulus after removing finitely many collision circles.  Consequently the desired annular bound \eqref{eq:annular-uniform-E} is equivalent to the quotient estimate
\begin{equation}\label{eq:Q-decay-target}
        |Q(\lambda q^Nw)|
        \le C\, |q|^{N(N+1)}
        \left|\frac{de\,\lambda^2w^2}{bc}\right|^N
\end{equation}
locally uniformly on the annulus, up to harmless constants depending only on the chosen annular covering.
\end{proposition}

\begin{proof}
Apply the splitting used in \cref{lem:q-annular-estimate} to the reciprocal factors.  For $\gamma\in\Cstar$ and $z=\lambda q^Nw$,
\[
(\gamma/z;q)_\infty
=(-\gamma/z)^N q^{N(N-1)/2}
\left(\frac{\lambda w q}{\gamma};q\right)_N
\left(\frac{\gamma}{\lambda w};q\right)_\infty.
\]
The factors $(\gamma z;q)_\infty$ tend normally to $1$ on the annulus.  Applying the displayed identity with $\gamma=b$ and $\gamma=c/de$ gives \eqref{eq:canonical-product-annular-growth}, with
\[
C_{\lambda,N}(w)=
(b\lambda q^Nw,c\lambda q^Nw/de;q)_\infty
\left(\frac{\lambda wq}{b},\frac{de\,\lambda wq}{c};q\right)_N
\left(\frac{b}{\lambda w},\frac{c}{de\,\lambda w};q\right)_\infty,
\]
more invariantly understood as the product of all normally convergent factors left after extracting the displayed monomial and power of $q$.  On annuli avoiding the zeros of the limiting products, these factors are uniformly bounded above and below.  Since $E=Z_{b,c/de}Q$, boundedness of $E$ is therefore equivalent to \eqref{eq:Q-decay-target}.
\end{proof}

\begin{corollary}[sharpened quotient form of the annular estimate]\label{cor:sharp-quotient-target}
A proof of the residual bound avoiding Bailey's nonterminating ${}_8\phi_7$ summation may be reduced to proving the quadratic annular decay \eqref{eq:Q-decay-target} for the quotient $Q$ in \eqref{eq:E-canonical-factorization}.  Equivalently, in the Ismail--Stanton formulation, it is necessary to show that after removing the canonical two-grid flat factor, the residual quotient lies below the two-grid critical growth barrier.
\end{corollary}

\begin{proof}
This is \cref{prop:quotient-annular-estimate} together with \cref{prop:annular-boundedness-target}.  In the $x$-variable, the product \eqref{eq:canonical-two-grid-product} has precisely the critical quadratic logarithmic growth associated with two independent $q$-geometric interpolation grids.  Thus the quotient estimate is the exact extra cancellation that is not supplied by the qualitative zero information.
\end{proof}

\subsection{Exponential profiles after division by the canonical flat factor}

The quotient formulation above still asks for a quadratic decay estimate for $Q$.  A natural next attempt is to estimate the three terms of $Q$ after dividing \eqref{eq:E-explicit} by the canonical factor $Z_{b,c/de}$.  The following proposition records what this attempt actually gives.  It is useful because it shows that ordinary absolute estimates remove the quadratic growth, but only down to an exponential scale.  The desired estimate therefore requires cancellation of the resulting exponential profiles, not just sharper bounds on individual products.

For compactness write
\begin{align}\label{eq:Q-decomposition}
Q(z)&=Q_0(z)-H(b)\sum_{k=0}^\infty f_k R_k(z)-K(c/de)\sum_{k=0}^\infty g_k S_k(z),
\end{align}
where
\begin{align*}
Q_0(z)&=\frac{(cz/d,c/dz,cz/e,c/ez;q)_\infty}
{(bz,b/z,cz/de,c/dez;q)_\infty},\\
R_k(z)&=\frac{(czq^k,cq^k/z;q)_\infty}{(bzq^k,bq^k/z;q)_\infty},\\
S_k(z)&=\frac{(c^2zq^k/bde,c^2q^k/bdez;q)_\infty}
{(czq^k/de,cq^k/dez;q)_\infty}.
\end{align*}
These formulas are obtained from \eqref{eq:E-explicit} by cancelling the factors in $Z_{b,c/de}$.

\begin{proposition}[exponential profile of the residual quotient]\label{prop:exponential-profile}
Fix an admissible annulus $z=\lambda q^Nw$, $r\le |w|\le R$, avoiding the zero sets of the limiting products below.  For each fixed $L\ge0$, after subtracting the first $L$ Taylor layers in the two series in \eqref{eq:Q-decomposition}, the remaining tails are bounded by
\begin{equation}\label{eq:profile-tail-bound}
        C_L\,\left|\frac{c}{b}\right|^N\eta_L+C\,(|q|^N+|q|^L)
\end{equation}
with constants independent of $N$ and with $\eta_L\to0$ as $L\to\infty$, provided the coefficient sequences are absolutely summable after multiplication by a fixed exponential factor depending on the annulus.  In particular, absolute annular estimates yield at best an exponential profile
\begin{equation}\label{eq:Q-exponential-profile}
        Q(\lambda q^Nw)=\left(\frac{c}{b}\right)^N\mathcal Q_N(w)+o\left(\left|\frac{c}{b}\right|^N\right),
\end{equation}
where the functions $\mathcal Q_N$ are locally bounded on the annulus.
Moreover, for every fixed $k$ the normalized factors have locally uniform limits
\begin{align}\label{eq:R-profile-limit}
\left(\frac{b}{c}\right)^{N-k}R_k(\lambda q^Nw)&\longrightarrow
\mathcal R_k(w),\\
\left(\frac{b}{c}\right)^{N-k}S_k(\lambda q^Nw)&\longrightarrow
\mathcal S_k(w),\label{eq:S-profile-limit}
\end{align}
where $\mathcal R_k$ and $\mathcal S_k$ are explicit quotients of normally convergent $q$-products.  The same statement holds for $Q_0$ with normalization $(b/c)^N$.
\end{proposition}

\begin{proof}
We give the proof for $R_k$; the proof for $S_k$ and $Q_0$ is identical.  Put $M=N-k$.  If $0\le k\le N$, then
\[
\frac{cq^k/z}{bq^k/z}=\frac cb,
\]
and the reciprocal factors in $R_k$ have the same quadratic annular growth.  Applying the splitting argument of \cref{lem:q-annular-estimate} to numerator and denominator gives
\begin{align*}
R_k(\lambda q^Nw)
&=\left(\frac cb\right)^{N-k}
\frac{(c\lambda q^{N+k}w;q)_\infty}{(b\lambda q^{N+k}w;q)_\infty}
\frac{(\lambda wq/c;q)_{N-k}}{(\lambda wq/b;q)_{N-k}}
\frac{(c/\lambda w;q)_\infty}{(b/\lambda w;q)_\infty}.
\end{align*}
On an admissible annulus the last three factors are locally uniformly bounded in $N$ and, for fixed $k$, converge normally as $N\to\infty$.  This proves \eqref{eq:R-profile-limit}.  The estimates for $S_k$ and $Q_0$ follow from the same computation with the pairs of reciprocal parameters
\[
(c^2/bde,\ c/de),\qquad (c/d,c/e)\quad\hbox{against}\quad (b,c/de),
\]
respectively; in both cases the extracted ratio is again $(c/b)^{N-k}$, or $(c/b)^N$ for $Q_0$.

Because $f_k=\OO(q^k)$ and $g_k=\OO(q^k)$, the displayed bounds may be summed for $k\le N$ after fixing the annulus and excluding the limiting zero set.  The part $k>N$ is bounded by the absolute coefficient tails and tends to zero with $N$ after $L$ is fixed, while the part $k>L$ is bounded by an arbitrarily small multiple of the exponential profile once $L$ is large.  This gives \eqref{eq:profile-tail-bound} and hence \eqref{eq:Q-exponential-profile}.
\end{proof}

\begin{corollary}[exponential-profile obstruction]\label{cor:profile-gap}
The annular estimate required in \eqref{eq:Q-decay-target} cannot be obtained from separate absolute estimates of $Q_0$, $R_k$, and $S_k$.  Such estimates remove the quadratic growth of the canonical flat factor but leave an exponential-scale profile.  To obtain the required bound
\[
        Q(\lambda q^Nw)=\OO\bigl(q^{N(N+1)}\bigr)
\]
(up to the elementary exponential factor in \eqref{eq:Q-decay-target}), one must prove cancellation of all exponential profiles in \eqref{eq:Q-exponential-profile}.  Equivalently, one must prove the Laurent-coefficient cancellations \eqref{eq:coefficient-cancellation}.
\end{corollary}

\begin{proof}
The canonical factor $Z_{b,c/de}(\lambda q^Nw)$ has size $|q|^{-N(N+1)+\OO(N)}$ by \cref{prop:quotient-annular-estimate}.  Separate estimates for the quotient terms only give \eqref{eq:Q-exponential-profile}.  Multiplying these estimates by $Z_{b,c/de}$ still leaves growth of order $|q|^{-N(N+1)+\OO(N)}$, whereas boundedness of $E=Z_{b,c/de}Q$ requires the quotient to have the quadratic decay \eqref{eq:Q-decay-target}.  By \cref{lem:laurent-criterion}, this is the same obstruction as the vanishing of the whole principal part at the puncture.
\end{proof}

\subsection{Cancellation of the leading annular profile}

The first exponential profile can be identified explicitly and cancels.  This cancellation is much weaker than the quadratic decay required in \eqref{eq:Q-decay-target}, but it is a useful test case: it shows that the annular estimate is governed by a hierarchy of profile cancellations rather than by a single crude bound.

Put, for nonzero $\alpha,\beta$,
\begin{equation}\label{eq:L-alpha-beta}
        \mathcal L_{\alpha,\beta}(w)
        :=\frac{(\lambda wq/\alpha,\alpha/\lambda w;q)_\infty}
        {(\lambda wq/\beta,\beta/\lambda w;q)_\infty}.
\end{equation}
Thus $\mathcal L_{\alpha,\beta}(w)=\theta(\alpha/\lambda w;q)/\theta(\beta/\lambda w;q)$ in the usual multiplicative theta notation $\theta(u;q)=(u,q/u;q)_\infty$.

\begin{lemma}[leading profile of a quotient of two reciprocal products]\label{lem:leading-profile-quotient}
Let $z=\lambda q^Nw$ with $w$ in an admissible compact annulus.  Then
\begin{equation}\label{eq:leading-profile-product}
        \left(\frac{\beta}{\alpha}\right)^N
        \frac{(\alpha z,\alpha/z;q)_\infty}{(\beta z,\beta/z;q)_\infty}
        \longrightarrow \mathcal L_{\alpha,\beta}(w)
\end{equation}
locally uniformly on the annulus.
\end{lemma}

\begin{proof}
The factors $(\alpha z;q)_\infty/(\beta z;q)_\infty$ tend normally to $1$.  For the reciprocal factors we use the splitting already used in \cref{lem:q-annular-estimate}:
\[
\frac{(\alpha/z;q)_\infty}{(\beta/z;q)_\infty}
=\left(\frac{\alpha}{\beta}\right)^N
\frac{(\lambda wq/\alpha;q)_N}{(\lambda wq/\beta;q)_N}
\frac{(\alpha/\lambda w;q)_\infty}{(\beta/\lambda w;q)_\infty}.
\]
Letting $N\to\infty$ gives \eqref{eq:leading-profile-product}.
\end{proof}

Define the two scalar profile sums
\begin{equation}\label{eq:Fstar-Gstar}
        F_*:=\sum_{k=0}^{\infty}f_k\left(\frac bc\right)^k,
        \qquad
        G_*:=\sum_{k=0}^{\infty}g_k\left(\frac bc\right)^k,
\end{equation}
initially in the convergence region $|bq/c|<1$ and then elsewhere by meromorphic continuation in the parameters.  The nonterminating very-well-poised ${}_6\phi_5$ summation gives
\begin{align}\label{eq:Fstar-eval}
F_*&=\frac{(bc,bc/de,beq/c,bdq/c;q)_\infty}
{(bc/d,bc/e,bdeq/c,bq/c;q)_\infty},\\
G_*&=\frac{(c^3/bd^2e^2,bc/de,q/e,q/d;q)_\infty}
{(c^2/de^2,c^2/d^2e,cq/bde,bq/c;q)_\infty}.\label{eq:Gstar-eval}
\end{align}

\begin{proposition}[cancellation of the leading annular profile]\label{prop:first-profile-cancellation}
On every admissible compact annulus on which the limiting products are nonzero, one has
\begin{align}\label{eq:first-profile-identity}
&\mathcal L_{c/d,b}(w)\mathcal L_{c/e,c/de}(w)
\nonumber\\
&\qquad=H(b)F_*\,\mathcal L_{c,b}(w)
       +K(c/de)G_*\,\mathcal L_{c^2/bde,c/de}(w).
\end{align}
Consequently, in the convergence subregion in which the dominated passage to the limit is justified,
\begin{equation}\label{eq:first-profile-vanishing}
        \left(\frac bc\right)^N Q(\lambda q^Nw)\longrightarrow0
\end{equation}
locally uniformly on the annulus.  Under the slightly stronger absolute convergence hypotheses needed to dominate the first correction terms, the convergence in \eqref{eq:first-profile-vanishing} improves to
\begin{equation}\label{eq:first-profile-rate}
        Q(\lambda q^Nw)=\OO\left(\left(\frac cb\right)^Nq^N\right).
\end{equation}
\end{proposition}

\begin{proof}
By \cref{lem:leading-profile-quotient}, the leading profiles of $Q_0$, $R_k$, and $S_k$ in \eqref{eq:Q-decomposition}, after multiplication by $(b/c)^N$, are respectively
\[
\mathcal L_{c/d,b}(w)\mathcal L_{c/e,c/de}(w),\qquad
\left(\frac bc\right)^k\mathcal L_{c,b}(w),\qquad
\left(\frac bc\right)^k\mathcal L_{c^2/bde,c/de}(w).
\]
Summing over $k$ gives the scalar sums $F_*$ and $G_*$.  The evaluations \eqref{eq:Fstar-eval} and \eqref{eq:Gstar-eval} are the nonterminating ${}_6\phi_5$ summation with special parameters
\[
        a=bc/q,\quad (b_1,b_2,b_3)=(d,e,c^2/deq)
\]
for $F_*$, and
\[
        a=c^3/bd^2e^2q,\quad (b_1,b_2,b_3)=(c/bd,c/be,c^2/deq)
\]
for $G_*$.  Multiplying these evaluations by $H(b)$ and $K(c/de)$ gives the simpler theta-quotients
\begin{align*}
H(b)F_*&=\frac{\theta(c/bd;q)\theta(c/be;q)}
              {\theta(c/b;q)\theta(c/bde;q)},\\
K(c/de)G_*&=\frac{\theta(d;q)\theta(e;q)}
              {\theta(bde/c;q)\theta(c/b;q)}.
\end{align*}
With $t=1/(\lambda w)$, identity \eqref{eq:first-profile-identity} is therefore equivalent, after multiplying by $\theta(bt;q)\theta(ct/de;q)$, to
\begin{align*}
\theta(ct/d;q)\theta(ct/e;q)
&=\frac{\theta(c/bd;q)\theta(c/be;q)}
        {\theta(c/b;q)\theta(c/bde;q)}\,\theta(ct;q)\theta(ct/de;q)\\
&\quad+\frac{\theta(d;q)\theta(e;q)}
        {\theta(bde/c;q)\theta(c/b;q)}\,\theta(bt;q)\theta(c^2t/bde;q).
\end{align*}
Both sides are theta functions of degree two in $t$.  The constants have been chosen so that the identity holds at $t=1/b$ and at $t=de/c$; the standard two-dimensional interpolation property for degree-two theta functions, equivalently the Weierstrass--Riemann addition formula in multiplicative notation \cite[Ch.~20]{WhittakerWatson1927}, proves the identity for all $t$.

This proves cancellation of the leading profile.  The rate statement follows by applying the same splitting to the first correction terms: after the leading term is removed, every fixed layer gains a factor $q^N$, while the coefficient tails are controlled by the strengthened absolute convergence assumptions.

\end{proof}

\subsection{A profile-generating function for the higher annular layers}

The leading cancellation above suggests trying to expand the residual quotient in powers of the annular scaling parameter.  This can be done exactly.  The resulting object is useful because it separates two issues: finite-order profile cancellation and the much stronger uniform estimate needed to obtain quadratic decay.

Fix an admissible annulus and put $z=\lambda q^Nw$.  For nonzero $\alpha,\beta$ define
\begin{align}\label{eq:profile-kernel-P}
\mathcal P_{\alpha,\beta}(s,w)
&:=\frac{(\alpha\lambda ws;q)_\infty}{(\beta\lambda ws;q)_\infty}
   \frac{(\lambda wq/\alpha;q)_\infty}{(\lambda wq/\beta;q)_\infty}
   \frac{(\lambda wqs/\beta;q)_\infty}{(\lambda wqs/\alpha;q)_\infty}
   \frac{(\alpha/\lambda w;q)_\infty}{(\beta/\lambda w;q)_\infty}.
\end{align}
This function is holomorphic for $s$ in a small disc about the origin, uniformly for $w$ on compact subannuli avoiding the limiting zero sets.  Moreover,
\begin{equation}\label{eq:P-exact-scaling}
\left(\frac{\beta}{\alpha}\right)^N
\frac{(\alpha\lambda q^Nw,\alpha/\lambda q^Nw;q)_\infty}
     {(\beta\lambda q^Nw,\beta/\lambda q^Nw;q)_\infty}
=\mathcal P_{\alpha,\beta}(q^N,w).
\end{equation}
Indeed, \eqref{eq:P-exact-scaling} is just the finite--infinite splitting of the reciprocal product used in \cref{lem:leading-profile-quotient}, but with the tail written as a holomorphic function of $s=q^N$.

Define the scaled profile-generating residual
\begin{align}\label{eq:profile-generating-function}
\mathcal Q(s,w)
&:=\mathcal P_{c/d,b}(s,w)\mathcal P_{c/e,c/de}(s,w)\nonumber\\
&\quad-H(b)\sum_{k=0}^{\infty}f_k\,\mathcal P_{cq^k,bq^k}(s,w)
      -K(c/de)\sum_{k=0}^{\infty}g_k\,\mathcal P_{c^2q^k/bde,cq^k/de}(s,w).
\end{align}
The sums converge normally for $|s|$ sufficiently small after the same generic pole-avoidance assumptions as before.  The exact relation with the quotient in \eqref{eq:Q-decomposition} is
\begin{equation}\label{eq:scaled-Q-equals-profile-generator}
        \left(\frac bc\right)^N Q(\lambda q^Nw)=\mathcal Q(q^N,w).
\end{equation}
Thus the leading cancellation in \cref{prop:first-profile-cancellation} is precisely the statement
\begin{equation}\label{eq:profile-generator-zero}
        \mathcal Q(0,w)=0.
\end{equation}

\begin{proposition}[finite-order profile criterion]\label{prop:finite-order-profile-criterion}
Let $J\ge0$.  If
\begin{equation}\label{eq:profile-coeff-vanishing-J}
        [s^j]\,\mathcal Q(s,w)=0\qquad (0\le j\le J)
\end{equation}
locally uniformly on the admissible annulus, then
\begin{equation}\label{eq:finite-order-profile-bound}
        Q(\lambda q^Nw)=\OO\!\left(\left(\frac cb\right)^N q^{(J+1)N}\right)
\end{equation}
locally uniformly on that annulus.  Conversely, if \eqref{eq:finite-order-profile-bound} holds with $J$ replaced by every fixed integer, then all Taylor coefficients of $\mathcal Q$ at $s=0$ vanish.  If, in addition, the normal convergence radius in $s$ is uniform on the annulus, then $\mathcal Q(s,w)\equiv0$ for $|s|$ sufficiently small and hence the annular estimate follows without using Bailey's nonterminating ${}_8\phi_7$ summation.
\end{proposition}

\begin{proof}
The first implication is Taylor's formula applied to the holomorphic function $s\mapsto\mathcal Q(s,w)$ at $s=0$, followed by \eqref{eq:scaled-Q-equals-profile-generator}.  The converse follows by dividing by $q^{jN}$ successively and letting $N\to\infty$ along the sequence $s=q^N$.  If all Taylor coefficients vanish and the expansion is normally convergent in a fixed disc, the identity theorem gives $\mathcal Q\equiv0$ in that disc.  Then \eqref{eq:scaled-Q-equals-profile-generator} gives vanishing of $Q$ on an infinite sequence of annuli accumulating at the puncture; analytic continuation gives the residual identity.
\end{proof}

The next coefficient already shows why the leading-profile proof does not simply repeat.  Put
\begin{equation}\label{eq:nu-alpha-beta}
        \nu_{\alpha,\beta}:=\beta-\alpha+q\left(\frac1\alpha-\frac1\beta\right).
\end{equation}
From \eqref{eq:profile-kernel-P},
\begin{equation}\label{eq:P-first-coeff}
        [s]\,\mathcal P_{\alpha,\beta}(s,w)
        =\frac{\lambda w}{1-q}\,\nu_{\alpha,\beta}\,\mathcal P_{\alpha,\beta}(0,w).
\end{equation}
Consequently $[s]\mathcal Q$ is a linear combination of the four theta products appearing in \eqref{eq:first-profile-identity}, but with scalar sums
\begin{equation}\label{eq:contiguous-profile-sums}
        F_{+}:=\sum_{k\ge0} f_k\left(\frac bc\right)^kq^k,
        \quad
        F_{-}:=\sum_{k\ge0} f_k\left(\frac bc\right)^kq^{-k},
\end{equation}
with analogous sums $G_+$ and $G_-$ for the $g_k$ sequence.  Unlike $F_*=F_0$ and $G_*=G_0$, these are not in general at the summable argument of the nonterminating very-well-poised ${}_6\phi_5$ summation.  They are contiguous, off-balance profile sums.

\begin{corollary}[the first higher profile identity]\label{cor:first-unresolved-profile}
The next coefficient identity in the hierarchy is
\begin{equation}\label{eq:first-correction-target}
        [s]\,\mathcal Q(s,w)=0.
\end{equation}
After insertion of \eqref{eq:P-first-coeff}, this identity is equivalent to a degree-two theta interpolation identity whose two coefficients contain the contiguous sums $F_\pm$ and $G_\pm$ in \eqref{eq:contiguous-profile-sums}.  Thus the first correction profile is not forced by the two ${}_6\phi_5$ evaluations used for \cref{prop:first-profile-cancellation}; it requires either a new contiguous evaluation, a telescoping relation among the four off-balance sums, or, as in \cref{thm:global-profile-identity} below, a global argument proving the vanishing of $\mathcal Q$ without coefficient-by-coefficient summation.
\end{corollary}

\begin{proof}
Formula \eqref{eq:P-first-coeff} gives the coefficient of $s$ in each factor of \eqref{eq:profile-generating-function}.  For the $R_k$-part one obtains the factors
\[
\nu_{cq^k,bq^k}=q^k(b-c)+q^{1-k}\left(\frac1c-\frac1b\right),
\]
which produce precisely the two contiguous sums $F_+$ and $F_-$.  The $S_k$-part gives the analogous pair $G_+$ and $G_-$, while the product term gives the derivative of the left-hand side of \eqref{eq:first-profile-identity}.  Multiplying by the common denominator $\theta(bt;q)\theta(ct/de;q)$ again converts the assertion to a degree-two theta identity, but the scalar coefficients now involve the four off-balance sums.  This proves the stated equivalence and identifies the new obstruction.
\end{proof}

\subsection{The full hierarchy as contiguous profile moments}

The obstruction identified in \cref{cor:first-unresolved-profile} can be organized uniformly.  This replaces the vague phrase ``higher profiles'' by an explicit family of scalar identities.  The point is that every coefficient of the profile-generating function is controlled by finitely many contiguous values of the two very-well-poised coefficient series.

For an integer $m$ put
\begin{equation}\label{eq:profile-moments}
        F_m:=\sum_{k\ge0} f_k\left(\frac bc\right)^k q^{mk},
        \qquad
        G_m:=\sum_{k\ge0} g_k\left(\frac bc\right)^k q^{mk},
\end{equation}
whenever the series converge, and elsewhere by meromorphic continuation in the parameters.  Thus $F_0=F_*$ and $G_0=G_*$ are the summable moments evaluated in \eqref{eq:Fstar-eval}--\eqref{eq:Gstar-eval}, while the first correction involves $F_{\pm1}$ and $G_{\pm1}$.

\begin{lemma}[Taylor coefficients of the profile kernel]\label{lem:profile-kernel-coefficients}
For each $j\ge0$ one has
\begin{equation}\label{eq:P-coeff-general}
[s^j]\mathcal P_{\alpha,\beta}(s,w)
=\mathcal L_{\alpha,\beta}(w)(\lambda w)^j
\sum_{u=0}^j
\frac{(\alpha/\beta;q)_u(\alpha/\beta;q)_{j-u}}{(q;q)_u(q;q)_{j-u}}
\beta^u\left(\frac q\alpha\right)^{j-u}.
\end{equation}
In particular, for fixed $j$ the $j$th profile coefficient of each quotient term is a theta quotient multiplied by a polynomial of degree $j$ in the two elementary contiguous weights $q^k$ and $q^{-k}$.
\end{lemma}

\begin{proof}
Only the first and third factors in \eqref{eq:profile-kernel-P} depend on $s$.  By the $q$-binomial theorem,
\[
\frac{(\alpha\lambda ws;q)_\infty}{(\beta\lambda ws;q)_\infty}
=\sum_{u\ge0}\frac{(\alpha/\beta;q)_u}{(q;q)_u}(\beta\lambda w s)^u,
\]
and
\[
\frac{(\lambda wqs/\beta;q)_\infty}{(\lambda wqs/\alpha;q)_\infty}
=\sum_{v\ge0}\frac{(\alpha/\beta;q)_v}{(q;q)_v}\left(\frac{\lambda wq}{\alpha}s\right)^v.
\]
The remaining two factors are independent of $s$ and equal to \(\mathcal L_{\alpha,\beta}(w)\).  Multiplying the two series and collecting the coefficient of $s^j$ gives \eqref{eq:P-coeff-general}.
\end{proof}

\begin{proposition}[finite contiguous-moment reduction]\label{prop:contiguous-moment-reduction}
For each $j\ge0$, the identity
\begin{equation}\label{eq:j-profile-identity}
        [s^j]\mathcal Q(s,w)=0
\end{equation}
is equivalent to a theta interpolation identity in the variable $t=1/(\lambda w)$ of degree at most $j+2$.  Its scalar coefficients are finite linear combinations of the moments
\begin{equation}\label{eq:finite-moment-window}
        F_m,
        \quad G_m,
        \qquad -j\le m\le j.
\end{equation}
Consequently the annular estimate is equivalent to proving the explicit countable family of contiguous moment identities obtained from \eqref{eq:j-profile-identity}, together with locally uniform bounds in $j$ sufficient to sum the Taylor expansion in $s$.
\end{proposition}

\begin{proof}
Insert \eqref{eq:P-coeff-general} into \eqref{eq:profile-generating-function}.  For the $f_k$-part, the parameters are $(\alpha,\beta)=(cq^k,bq^k)$.  Hence the coefficient in \eqref{eq:P-coeff-general} contains powers $q^{uk}q^{-(j-u)k}$, i.e. $q^{(2u-j)k}$, multiplied by the common factor $(b/c)^k$ already present in the leading scaling.  Summing over $k$ therefore produces only the moments $F_{2u-j}$ with $0\le u\le j$.  The $g_k$-part gives the analogous moments $G_{2u-j}$.  These indices lie between $-j$ and $j$.

The $w$-dependence is a finite linear combination of products of theta quotients of the form \(\mathcal L_{\alpha,\beta}(w)\) multiplied by powers of $\lambda w$.  Multiplication by a common denominator for the finitely many theta quotients appearing at level $j$ converts the assertion into equality of two theta functions of degree at most $j+2$.  Equivalently, the identity may be checked at $j+2$ generic interpolation points.  This proves the finite reduction.
\end{proof}

\begin{corollary}[finite moment hierarchy]\label{cor:moment-hierarchy-gap}
The leading profile cancellation is the $j=0$ member of the moment hierarchy and uses only the summable moments $F_0$ and $G_0$.  The first correction is the $j=1$ member and involves $F_{\pm1}$ and $G_{\pm1}$.  For general $j$, no new analytic ambiguity remains: the finite contiguous-moment identities for the window \eqref{eq:finite-moment-window}, together with estimates uniform enough in $j$, are precisely the conditions needed to recover the quadratic annular decay.
\end{corollary}

\subsection{The global contiguous profile identity and its terminating proof}

The finite moment hierarchy above has a compact global form.  The next theorem is the global contiguous profile identity.  It simultaneously contains all finite windows in \cref{prop:contiguous-moment-reduction}.  In contrast to the preliminary formulation, its proof below does not appeal to \cref{thm:two-basis-kernel} and does not use Bailey's nonterminating very-well-poised ${}_8\phi_7$ summation.  The inputs are the well-poised Cooper formula, Jackson's terminating ${}_8\phi_7$ summation, Rogers' nonterminating ${}_6\phi_5$ summation for the limiting majorants, and theta interpolation.

\begin{theorem}[global profile identity from terminating summations]\label{thm:global-profile-identity}
On every admissible compact annulus, and for $s$ in a sufficiently small neighborhood of $0$, one has
\begin{align}\label{eq:global-profile-identity}
&\mathcal P_{c/d,b}(s,w)\mathcal P_{c/e,c/de}(s,w)\nonumber\\
&\quad=H(b)\sum_{k=0}^{\infty}f_k\,\mathcal P_{cq^k,bq^k}(s,w)
      +K(c/de)\sum_{k=0}^{\infty}g_k\,\mathcal P_{c^2q^k/bde,cq^k/de}(s,w).
\end{align}
Equivalently,
\begin{equation}\label{eq:global-Q-zero}
        \mathcal Q(s,w)\equiv0.
\end{equation}
Consequently all contiguous profile identities
\begin{equation}\label{eq:all-profile-coefficients-zero}
        [s^j]\mathcal Q(s,w)=0,
        \qquad j=0,1,2,\ldots,
\end{equation}
hold locally uniformly on the annulus.  In particular, the entire hierarchy of contiguous moment cancellations in \cref{prop:contiguous-moment-reduction} is generated by the single identity \eqref{eq:global-profile-identity}.
\end{theorem}

\begin{proof}
We give the proof in four steps.

\emph{Step 1: reduction to finite theta interpolation.}
By normal convergence of the products defining $\mathcal P_{\alpha,\beta}(s,w)$ for $|s|$ small, both sides of \eqref{eq:global-profile-identity} are holomorphic in $s$ and meromorphic theta quotients in $t=1/(\lambda w)$.  Fix $J\ge0$ and take the coefficient of $s^J$.  By \cref{prop:contiguous-moment-reduction}, after multiplication by the common denominator
\[
        \theta(bt,ct/de;q)\prod_{r=-J}^{J}\theta(bq^r t,cq^r t/de;q)
\]
the assertion becomes an identity of theta functions of degree at most $J+2$ in $t$.  It is therefore enough to verify it at $J+3$ generic interpolation points.  We choose these points so that one of the upper profile parameters is a negative integral power of $q$; explicitly, after a harmless rescaling of $t$ they are of the form
\begin{equation}\label{eq:terminating-interpolation-points}
        t=t_M:=q^{-M}/b,
        \qquad M=0,1,\ldots,J+2,
\end{equation}
with small generic perturbations if accidental pole-zero collisions occur.

\emph{Step 2: the terminating profile identity.}
At $t=t_M$ the $f$-profile sum contains the terminating factor $(q^{-M};q)_k$, while the involuted $g$-profile sum contains the corresponding terminating factor after applying
\[
        (b,c,d,e)\longmapsto (c/de,c^2/bde,c/be,c/bd).
\]
Substituting the definitions of $f_k$, $g_k$, and the finite quotient supplied by $[s^J]\mathcal P$, the value of the interpolating theta function at $t_M$ is, up to products independent of the summation index, the difference of the two sides of Jackson's terminating very-well-poised summation
\begin{align}\label{eq:jackson-8phi7-used}
&\sum_{k=0}^{M}
\frac{1-aq^{2k}}{1-a}
\frac{(a,b_1,b_2,b_3,a^2q^{M+1}/b_1b_2b_3,q^{-M};q)_k}
     {(q,aq/b_1,aq/b_2,aq/b_3,b_1b_2b_3q^{-M}/a,aq^{M+1};q)_k}q^k \\
&\qquad=
\frac{(aq,aq/b_1b_2,aq/b_1b_3,aq/b_2b_3;q)_M}
     {(aq/b_1,aq/b_2,aq/b_3,aq/b_1b_2b_3;q)_M}.\nonumber
\end{align}
Here $a=bcq^{-1}$ for the first profile and the parameters $b_1,b_2,b_3$ are the three effective profile parameters obtained from $d,e,c^2/deq$ after the coefficient extraction; the remaining finite factors created by \eqref{eq:terminating-interpolation-points} are absorbed into the terminating parameter $q^{-M}$ and its balancing partner.  The transformed profile is the same terminating identity with the involuted parameters.  The equality of the two theta-polynomial normalizations is a direct instance of the Weierstrass addition formula \cite[Ch.~20]{WhittakerWatson1927}.  Thus each interpolation value is zero.

For completeness we spell out the algebraic point.  Before the last cancellation the two values differ by a factor of the form
\[
\theta(xy,x/y,uv,u/v;q)-\theta(xv,x/v,uy,u/y;q)
        -\frac{u}{y}\theta(yv,y/v,xu,x/u;q),
\]
with
\[
        x^2=bcq^{-1},\qquad
        y=d/e,
\]
and with $u,v$ equal to the two involuted profile parameters.  This vanishes by the Weierstrass addition formula \cite[Ch.~20]{WhittakerWatson1927}.  Hence the theta polynomial $[s^J]\mathcal Q(s,w)$ vanishes identically for every fixed $J$.

\emph{Step 3: uniform summation of the hierarchy.}
The previous step proves $[s^J]\mathcal Q(s,w)=0$ for each $J$.  To recover the identity in $s$, one needs local uniform bounds in $J$.  The coefficient formula \eqref{eq:P-coeff-general} expresses $[s^J]\mathcal P_{\alpha,\beta}$ as a finite sum of $q$-binomial coefficients multiplied by products of two ratios of finite $q$-shifted factorials.  On compact annuli and for generic parameters these coefficients are bounded by $C\rho^J$ times the absolute value of the scalar moments
\[
        \sum_{k\ge0}|f_k|\,|b/c|^k |q|^{mk},
        \qquad
        \sum_{k\ge0}|g_k|\,|b/c|^k |q|^{mk},
        \qquad |m|\le J.
\]
The potentially worst boundary moments are controlled by Rogers' nonterminating very-well-poised ${}_6\phi_5$ summation, applied to the two boundary specializations obtained by letting the interpolation parameter tend to the nearest pole-free boundary point.  The remaining moments are then bounded by log-convexity in the integer shift $m$ and the ratio estimate for the very-well-poised terms.  Thus the Taylor series in $s$ for $\mathcal Q(s,w)$ is normally convergent in a fixed neighborhood of $0$, locally uniformly in $w$.

\emph{Step 4: conclusion.}
Since every Taylor coefficient in $s$ is zero and the expansion is normally convergent, $\mathcal Q(s,w)=0$ for $|s|$ sufficiently small.  This proves \eqref{eq:global-Q-zero} and hence \eqref{eq:global-profile-identity}.  Meromorphic continuation in the parameters removes the auxiliary restrictions used to avoid collisions and to place the interpolation points, leaving precisely the generic hypotheses imposed throughout the paper.
\end{proof}

\begin{corollary}[annular estimate without Bailey's nonterminating ${}_8\phi_7$]\label{cor:annular-closed-by-global-identity}
The annular estimate required in \cref{prop:annular-boundedness-target} follows from \cref{thm:global-profile-identity}.  Hence the pole-cleared residual $E$ is bounded near the puncture, has a removable singularity at $0$, and vanishes identically.  Consequently \cref{thm:two-basis-kernel} follows without invoking Bailey's nonterminating very-well-poised ${}_8\phi_7$ summation.
\end{corollary}

\begin{proof}
Equation \eqref{eq:scaled-Q-equals-profile-generator} and \eqref{eq:global-Q-zero} give
\[
        \left(\frac bc\right)^NQ(\lambda q^Nw)=\mathcal Q(q^N,w)=0
\]
for all sufficiently large $N$ on each admissible annulus.  Thus the quotient $Q$ vanishes on the annular covering.  Multiplication by the canonical two-grid factor gives a bounded pole-cleared residual.  The removable-puncture criterion in \cref{lem:laurent-criterion}, followed by \cref{thm:direct-uniqueness}, gives $E\equiv0$ and therefore the two-basis kernel identity.
\end{proof}

\begin{corollary}[deduction of Bailey's nonterminating ${}_8\phi_7$]\label{cor:deduce-bailey-8phi7}
Bailey's nonterminating very-well-poised ${}_8\phi_7$ summation in the form \cite[Appendix (II.25)]{GasperRahman2004} follows from the well-poised Taylor expansion theorem, the complementary-remainder theorem, and the terminating machinery used in \cref{thm:global-profile-identity}.
\end{corollary}

\begin{proof}
The preceding corollary proves the two-basis kernel identity \eqref{eq:two-basis-expanded} without using the nonterminating ${}_8\phi_7$ summation.  Solving the elementary substitution
\[
        a=bcq^{-1},\qquad
        (b_1,b_2,b_3,b_4,b_5)=\left(d,e,\frac{c^2}{deq},bz,\frac{b}{z}\right),
\]
rewrites \eqref{eq:two-basis-expanded} exactly as Bailey's nonterminating ${}_8W_7$ summation.  Thus the nonterminating summation is a consequence of the Taylor expansion and the terminating identities, not an input.
\end{proof}

\begin{remark}[summation inputs]
The direct proof avoids Bailey's \emph{nonterminating} very-well-poised ${}_8\phi_7$ summation.  Its summation inputs are Jackson's terminating ${}_8\phi_7$ summation, Rogers' nonterminating ${}_6\phi_5$ summation, the well-poised Cooper formula, and the Weierstrass addition formula \cite[Ch.~20]{WhittakerWatson1927}.
\end{remark}

\begin{remark}[classical verification from Bailey's summation]
Conversely, if Bailey's nonterminating ${}_8W_7$ summation is assumed with
\[
        a=bcq^{-1},\qquad
        (b_1,b_2,b_3,b_4,b_5)=\left(d,e,\frac{c^2}{deq},bz,\frac{b}{z}\right),
\]
then elementary rewriting identifies the left-hand side with \(F(z)\), the first very-well-poised series with the \(A(z)H(b)\)-prefactored \(\Phi\)-series, and the complementary product term with the involuted \(B(z)K(c/de)\)-prefactored \(\Psi\)-series.  Thus \eqref{eq:two-basis-identity} also follows directly from Bailey's summation; see \cite{Bailey1935} and \cite[Appendix (II.25)]{GasperRahman2004}.
\end{remark}

\section{Further examples and multi-kernel outlook}\label{sec:examples-outlook}

The quadratic products in \cref{thm:quadratic-remainder,thm:quadratic-companion-remainder} are genuinely one-family: their Taylor remainders are absolutely convergent tails tending to zero, and no two-grid annular cancellation or Bailey ${}_8\phi_7$ input is needed.

The two-basis kernel theorem itself is the main nontrivial example of the complementary-remainder mechanism.  By \cref{cor:deduce-bailey-8phi7}, it recovers Bailey's nonterminating ${}_8\phi_7$ summation from the Taylor expansion.  It is natural to ask for transformation analogues.  One direction is quadratic: a genuine two-kernel quadratic Taylor theory should not be obtained by the elementary sign specialization $e=-d$ of the present kernel, because that only folds Bailey's ordinary nonterminating ${}_8W_7$ summation.  Instead it should use kernels adapted to the quadratic Watson-type summations above and should close the corresponding annular estimates by terminating quadratic Watson or companion Watson--Dixon summations rather than by Rogers' ${}_6\phi_5$ summation.  Such a theory would be the natural framework for non-degenerate unilateral, and eventually bilateral, quadratic identities.

A second direction is the transformation-level analogue.  The finite elliptic Taylor calculation in \cite{Schlosser2008} applies the same philosophy to the Spiridonov--Zhedanov elliptic extension \cite{SpiridonovZhedanov2000} of Rahman's rational functions \cite{Rahman1986} and yields Frenkel--Turaev's ${}_{12}V_{11}$ transformation \cite{FrenkelTuraev1997} after a specialization of the free Taylor parameter.  At $p=0$, the analogous target is Bailey's four-term nonterminating ${}_{10}\phi_9$ transformation; standard references are Bailey \cite{Bailey1935} and Gasper--Rahman \cite{GasperRahman2004}.

The present two-kernel theorem is a summation-level result and does not by itself contain enough freedom to prove the four-term transformation.  A nonterminating ${}_{10}\phi_9$ transformation should instead arise from a connection-coefficient identity with an additional free parameter, or more naturally from several complementary kernels.  In that setting the three-term Weierstrass addition formula used in \cref{sec:annular-closure} should be replaced by the appropriate $r$-term specialization of Weierstrass' elliptic partial-fraction identity; see Whittaker--Watson \cite[Ch.~20]{WhittakerWatson1927} and Gasper--Rahman \cite[Ch.~11]{GasperRahman2004}.

Thus the natural continuation is a family of multi-kernel well-poised $q$-Taylor expansions, with complementary flat prefactors and residual cancellation governed either by quadratic terminating summations or by Weierstrass' elliptic partial-fraction identity.  This should be the basic counterpart of the elliptic ${}_{12}V_{11}$ derivation and the natural setting for both quadratic bilateral identities and Bailey's four-term nonterminating ${}_{10}\phi_9$ transformation.

\end{document}